\documentclass[reqno,12pt,letterpaper]{amsart}
\usepackage[proof]{macros}

\numberwithin{equation}{section}
\definecolor{usuai}{RGB}{40, 148, 200}

\usepackage{cite}
\usepackage{graphicx}
\usepackage{pgfplots,tikz}
\usepackage{appendix}
\usepackage{enumitem}
\usepackage{mathtools}
\mathtoolsset{showonlyrefs}

\newcommand{\setr}{\mathbb{R}}

\newcommand{\setc}{\mathbb{C}}

\newcommand{\id}{\operatorname{Id}}

\newcommand{\bigo}{\mathcal{O}}

\newcommand{\fscinf}{C^\infty}

\newcommand{\ang}[1]{\left<#1\right>}
\newcommand{\cur}[1]{\left\{#1\right\}}

\newcommand{\brac}[1]{\left(#1\right)}

\newcommand{\Laplace}{\Delta}

\newcommand{\pd}{\partial}
\newcommand{\op}{\operatorname{Op}}
\newcommand{\oph}{\operatorname{Op}_h}

\newcommand{\cre}{\operatorname{Re}}

\newcommand{\abs}[1]{\lvert#1\rvert}

\begin{document}

\title[Exponential decay for damped Klein-Gordon equations]{Exponential decay for damped Klein-Gordon equations on asymptotically cylindrical and conic manifolds}

\author{Ruoyu P. T. Wang}
\address{Department of Mathematics, Northwestern University, 2033 Sheridan Road, Evanston, Illinois 60208, USA}

\email{rptwang@math.northwestern.edu}



\begin{abstract}
We study the decay of the global energy for the damped Klein-Gordon equation on non-compact manifolds with finitely many cylindrical and subconic ends up to bounded perturbation. We prove that under the Geometric Control Condition, the decay is exponential, and that under the weaker Network Control Condition, the decay is logarithmic, by developing the global Carleman estimate with multiple weights. 
\end{abstract}

\maketitle
\section{Introduction}
In this paper we study the decay of the global energy for the damped Klein-Gordon equation \eqref{S1L5}, on non-compact manifolds with finitely many ends of a wide class up to bounded perturbation, described in \eqref{SS12}, including asymptotically cylindrical and conic ends. We prove in Theorem \ref{LT1} that under the Geometric Control Condition given by Definition \ref{S1T2}, in which the average of damping along each geodesic is uniformly bounded from below, the global energy decays exponentially. We prove in Theorem \ref{LT2}, that under the Network Control Condition given by Definition \ref{S1T1}, in which each point in the space is within some uniform distance from the sufficient damped region, the global energy decays logarithmically. These results generalise those in \cite{bj16}. The main new tool is the Carleman estimates with multiple weights in Theorem \ref{S2T2}. 
\subsection{Geometric setting}\label{SS12}
Consider the model manifold $(M,g_0)$, a non-compact connected $d$-dimensional manifold without boundary, with $N$ infinite ends,
\begin{equation}\label{S1L9}
M=\overline{M_0}\cup \left(\bigcup_{k=1}^NM_k\right),
\end{equation}
where $\overline{M_0}$ is a compact, connected manifold with boundary $\pd \overline{M_0}=\bigsqcup_{k=1}^N \{1\}\times \pd M_k$. Denote the interior of $\overline{M_0}$ by $M_0$. Each end $M_k$ is identified as a cylinder $(1,\infty)_r\times\pd M_k$ endowed with a product metric
\begin{equation}\label{S1L3}
dr^2+\theta^2_k(r) h_k,
\end{equation}
where $\pd M_k$ is a $(d-1)$-dimensional compact manifold without boundary, $h_k$ a smooth metric on $\pd M_k$. The scaling functions $\theta_k\in C^\infty([1,\infty); \mathbb{R}_{> 0})$ satisfy either one of the following conditions:
\begin{gather}\label{S1L1}
\lim_{r\rightarrow \infty}\theta_k(r)=\infty,\quad\abs{\pd_r^m \theta_{k}}\le C_m<\infty, \forall m\ge 1; \textrm{or}\\\label{S1L2}
\theta_{k}\equiv 1.
\end{gather}
We call ends with the scaling functions $\theta_k$ in \eqref{S1L1} sub-conic ends, and ends with those in \eqref{S1L2} cylindrical ends. Specifically, sub-conic ends with the scaling function $\theta_k(r)=r$ are called conic ends. 

In this paper, we specify and work with bounded perturbations of the model metric \eqref{S1L3}. Specify a manifold of bounded geometry $(M, g)$ that is a bounded perturbation of our model manifold $(M,g_0)$, in the sense that both the identity map 
\begin{equation}\label{S1LL10}
\Phi_0: (M, g)\rightarrow (M, g_0); \quad p\mapsto p
\end{equation}
and its inverse
$\Phi_0^{-1}: (M, g_0)\rightarrow (M, g)$ are $C^\infty_b$-maps between two manifolds. See Appendix \ref{SA} for the definition of boundedness on manifolds of bounded geometry. Note that $d\Phi_0(p)$ is uniformly bounded at each $p\in (M, g)$, from both above and below as a map from $T_pM$ equipped with $g$ to $T_pM$ equipped with $g_0$. 

We inexhaustively list some examples that are compatible with our setting:
\begin{example}\label{S1LL15}
\begin{enumerate}[wide]
\item \label{S1L6}Euclidean spaces $\setr^d$, with $M_0$ being the unit open ball and $M_1$ being the rest of $\setr^d$ as $[1,\infty)_r\times\mathbb{S}^{d-1}$ in spherical coordinates. This is a conic end with $\theta_1(r)=r$.
\item \label{S1L7}Euclidean spaces $\setr^d$ as above, but endowed with a bounded perturbed metric, whose local matrix form in the canonical Euclidean coordinates, $g(x)$, and its inverse $g^{-1}(x)$, are smooth matrix-valued functions of which components are $C^\infty_b(\setr^d)$. 
\item Asymptotically conic manifolds, also as known as Riemannian scattering spaces, of finitely many ends of the form $M_k=[0,1]_x\times\pd M_k$ endowed with scattering metrics $x^{-4}dx^2+x^{-2}h_k$. Here $h_k$'s are smooth symmetric 2-cotensors on $M_k$ whose restriction to $\pd M_k$ is positive-definite. In our model we realise the metric as a bounded perturbation of $x^{-4}dx^2+x^{-2}h_k'=dr^2+r^2h_k'$ where $r=x^{-1}$ and $h_k'$'s are metrics on $\pd M_k$ independent of $r$. See \cite{mel95} for further details. 
\item \label{S1L8}Product cylinders of the form $(-\infty, \infty)\times\pd M$ where $\pd M$ is a closed manifold, by taking $M_0=(-1,1)\times\pd M$, $[1,\infty)\times M_1=\pd M$ and $M_2=(-\infty,-1]\times \pd M$. Or an one-ended cylinder glued to some closed manifold. More generally, asymptotically cylindrical manifolds also work with our setting. Those are manifolds with finitely many ends of the form $[0,1]_x\times\pd M_k$ endowed with $x^{-2}dx^2+h$. Here again $h$ is a smooth symmetric 2-cotensor on $M_k$ whose restriction to $\pd M_k$ is positive-definite. See \cite{mel95} for further details.  
\item \label{S1L4} Elliptic paraboloid, $\{(x,y,z):z=x^2+y^2\}\subset \setr^3$ with $M_0$ being the tip $\{z\le 1\}$ and $M_1$ be the rest of paraboloid as $[1,\infty)_r\times \mathbb{S}^1_\theta$ equipped with metric $(1+r^{-1}/4)dr^2+rd\theta^2$, under the change of coordinates $(x,y,z)=(r^{1/2}\cos{\theta}, r^{1/2}\sin{\theta}, r)$. Here this metric on $M_1$ is a bounded perturbation of $dr^2+rd\theta^2$, whose scaling function is $\theta_1(r)=r^{1/2}$, which is sub-conic. 
\item Boundedly perturbed cylinders. Consider the surface $\{(x,y,z): (2+\cos{z})^2=x^2+y^2\}\subset \setr^3$, which is realised as the bounded perturbation to $(-\infty, \infty)_r\times \mathbb{S}^1_\theta$ equipped with metric $dr^2+(2+\cos{r})d\theta^2$. Note that, this surface is not an asymptotically cylindrical manifold, as the $\cos{r}$ cannot be well-defined at the spatial infinity $r=\infty$. But we could still cope with this manifold as a bounded perturbation of the product cylinder. 
\item Any connected sum along balls of finitely many equidimensional ends of the types above.
\end{enumerate}
\end{example}
In this paper, we prove that in the geometric settings as above, one has exponential or logarithmic decays of the global energy for the damped Klein-Gordon equations by assuming suitable dynamical control conditions. It is also noted that, hyperbolic manifolds do not fit our analysis. 

\subsection{Damped Klein-Gordon equations}
Consider a damping function $a\in C_b^\infty(M)$, a smooth function on $M$ whose derivatives of all orders are bounded by uniform constants dependent only on the order. The damped Klein-Gordon equation on our manifold $(M,g)$ reads
\begin{equation}\label{S1L5}
	\begin{cases}
	\brac{\Laplace_{g}+\id+\pd_t^2+a\pd_t}u(t,x)=0, \quad\textrm{on }\setr_{t\ge 0}\times M_x\\
	u(0, x)=u_0(x)\in H^{2}(M), \quad \pd_t u(0, x)=u_1(x)\in H^{1}(M)
 \end{cases},
\end{equation}
where $\Delta_g$ is the positive Laplace-Beltrami operator on $(M,g)$. Consider
\begin{equation}
A=\begin{pmatrix}
0 & \id\\
-\left(\Delta_g+\id\right) & -a(x)
\end{pmatrix},
\end{equation}
which is a bounded linear operator from $D(A)=H^2(M)\times H^1(M)$ to $X=H^1(M)\times L^2(M)$, which is further dissipative in the sense that
\begin{equation}
\cre{\left\langle A\left(u, v\right), (u,v)\right\rangle}_X=-\int_M a(x)\abs{v(x)}^2~dg\le 0.
\end{equation}
By noting that $D(A)$ is dense in $X$, $A$ is a bounded dissipative operator on Hilbert space $X$, and the Lumer-Phillips theorem tells us $A$ generates a strongly continuous semigroup $e^{tA}$ on $X$ that is further a contraction semigroup, in the sense that $\|e^{tA}\|_{X\rightarrow X}\le 1$ for each $t\ge 0$. Note we can formulate the equation \eqref{S1L5} as a Cauchy problem for $U(t, x)=(u(t,x), \pd_t u(t,x))$, that is
\begin{equation}
\pd_t U(t,x)=A(x)U(t,x),\quad U(0,x)=U_0(x)=(u_0(x), u_1(x))\in X,
\end{equation}
and the strongly continuous semigroup $e^{tA}$ is the solution operator to the Cauchy problem, where the unique solution is $e^{tA}U_0$. As we look into how fast the global energy decays, it suffices to look at the decay of the operator norm of the semigroup $e^{tA}$. Indeed, the energy of the solution to \eqref{S1L5} is
\begin{equation}
E(u, t)=\frac12\int_M \abs{\nabla_x u(t,x)}^2+\abs{\pd_t u(t,x)}^2~dx\le\frac12\left\|e^{tA}\right\|_{X\rightarrow X}\left\|(u_0, u_1)\right\|_X.
\end{equation}
The semigroup $e^{tA}$ weakly decays to 0 when the damping $a$ is smooth and not zero somewhere on $M$, as in \cite{wal77}. In this paper, we are interested in two types of decays of the semigroup $e^{tA}$: the exponential decay, which is the fastest decay one expects in the contexts of a smooth and bounded damping, and the logarithmic decay, which is a kind of non-uniform decay under some weak dynamical hypotheses. 

About the damped Klein-Gordon equation and the damped wave equation, there have been many results known when $M$ is compact and the damping is smooth. It is known that the exponential decay of the semigroup is equivalent to the Geometric Control Condition, which is a dynamical hypothesis that all trajectories of the Hamiltonian flow intersect the support of the damping $a(x)$, as in \cite{rt74,blr88,blr92,bg97}. In \cite{leb93} it was shown that there is a logarithmic decay as long as the damping is non-trivial. It is also noted that other non-uniform stability properties have been actively investigated, as in, inexhaustively listed here, \cite{aln14,bh07,csvw14,bc15}. 

However the picture is less complete for exponential results on non-compact manifolds without boundary. The fundamental result in \cite{bj16} generalises the Geometric Control Condition to $\setr^d$, with a uniform lower bound of the average of damping along the Hamiltonian flow. It was also shown that the Geometric Control Condition gives exponential decay of the semigroup, and that there is logarithmic decay when another dynamical hypothesis, called the Network Control Condition, is imposed. In \cite{wun17} a polynomial decay was shown via Schr\"{o}dinger observability for a periodic damping on $\setr^d$ under no further dynamical assumptions. In \cite{jr18}, the sharp polynomial global energy decay for the damped wave equation on $\setr^d$ with an asymptotically periodic damping was shown. In \cite{roy18a} the results of \cite{bj16,wun17} were extended to highly oscillatory periodic dampings. See also, inexhaustively listed here, \cite{mr18,ms18,gjm19,gre19,cpsst19} for recent development on Euclidean spaces. 

The purpose of this paper is to extend the results of \cite{bj16} to a wider class of open manifolds, namely $(M, g)$ prespecified in \eqref{S1LL10}. In \cite{bj16}, the results have been shown for the Euclidean cases \eqref{S1L6}, \eqref{S1L7} of Example \ref{S1LL15}. The possibility of proving such results on product cylinders as in \eqref{S1L8} was also hinted. Our paper generalises their results to manifolds with cylindrical and sub-conic ends. Here we define the Geometric Control Condition on the prespecified manifold $(M,g)$: 
\begin{definition}[Geometric Control Condition]\label{S1T2}
We say the damping $a$ satisfies the Geometric Control Condition $(T,\alpha)$ on $(M, g)$, for $T, \alpha>0$, if for $(x,\xi)\in \Sigma$, where $\Sigma=\{(x,\xi)\in T^*M: \abs{\xi}^2=1\}$, one has
\begin{equation}
\ang{a}_T(x,\xi)=\frac{1}T\int_0^T \left(\left(\Pi_x\circ\varphi_t\right)^*a\right)(x,\xi)~dt\ge \alpha>0,
\end{equation}
where $\varphi_t$ is the Hamiltonian flow associated with $\abs{\xi}_{g}^2$, and $\Pi_x$ is the projection from fibres of $T^*M$ to the base variable. 
\end{definition}
We claim the first main result that Geometric Control Condition gives exponential decay of the semigroup $e^{tA}$: 
\begin{theorem}[Exponential decay of energy]\label{LT1}
Assume $a\in \fscinf_b(M)$ where $a\ge 0$ everywhere, satisfies the Geometric Control Condition $(T,\alpha)$, then the semigroup $e^{tA}$ decays exponentially in the sense that
$\|e^{tA}\|_{X\rightarrow X}\le Me^{-\lambda t}$, for each $t\ge 0$, for some $M, \lambda>0$. It is then implied that the solution $u$ to the damped Klein-Gordon equation with initial datum $(u_0, u_1)\in H^{2}(M)\times H^1(M)$, 
\begin{equation}
	\begin{cases}
	\brac{\Laplace_{g}+\id+\pd_t^2+a\pd_t}u(t,x)=0, \quad\textrm{on }\setr_{t\ge 0}\times M_x\\
	u(0, x)=u_0(x)\in H^{2}(M), \quad \pd_t u(0, x)=u_1(x)\in H^{1}(M)
 \end{cases}.
\end{equation}
decays exponentially, in the sense that there exists $C, \lambda> 0$, 
\begin{equation}
E(u, t)\le \frac12 Ce^{-\lambda t}\left(\left\|u_0\right\|_{H^1(M)}^2+\left\|u_1\right\|_{L^2(M)}^2\right)^{1/2}.
\end{equation}
\end{theorem}
The decay of the global energy for the damped Klein-Gordon equation is determined by how fast the high frequency waves and the low frequency waves decay. The energy of the high frequency waves semiclassically concentrates near the Hamiltonian flow. This phenomenon hints at why the Geometric Control Condition plays an important role here. On the other hand, the low frequency waves do not concentrate. But as a result of their long wavelengths, they can see the damping from a distance even if many trajectories do not encounter the damping. However, the sparser the damping is, the weaker the decay gets. Therefore to obtain a uniform rate of decay we do not want to be too far away from the damping. This inspires the following dynamical hypothesis. 
\begin{definition}[Network Control Condition]\label{S1T1}
We say the damping $a$ satisfies the Network Control Condition $(L, \omega, 2\beta, \{x_n\})$ on $M$, for $L, \omega, 2\beta>0$, and $\{x_n\}$ a set of points on $M$, if at each $x\in M$,
\begin{equation}
d(x, \bigcup_n \left\{x_n\right\})\le L,
\end{equation}
and $a(x)\ge 2\beta>0$ on $\bigcup_n B(x_n,\omega)$. 
\end{definition}
This hypothesis has been introduced in \cite{bj16}, and gives logarithmic decay on $\setr^d$. The logarithmic decay on compact manifolds has also been considered in \cite{leb93,lr97}. Here is our second result: 
\begin{theorem}[Logarithmic decay of energy]\label{LT2}
Assume $a\in \fscinf_b(M)$ where $a\ge 0$ everywhere, satisfies the Network Control Condition $(L, r, 2\beta, \{x_n\})$, then for each $k\ge 1$, the solution $u$ to the damped Klein-Gordon equation with initial datum $(u_0, u_1)\in H^{k+1}(M)\times H^k(M)$, 
\begin{equation}
	\begin{cases}
	\brac{\Laplace_{g}+\id+\pd_t^2+a\pd_t}u=0, \quad\textrm{on }\setr_{t\ge 0}\times M_x\\
	u|_{t=0}=u_0\in H^{k+1}(M), \quad \pd_t u|_{t=0}=u_1\in H^{k}(M)
 \end{cases}
\end{equation}
decays logarithmically, in the sense that there exists $C_k> 0$, 
\begin{equation}\label{S1LL16}
E(u)=\left(\left\|\nabla_g u(t)\right\|^2+\left\|\pd_t u(t)\right\|^2\right)^{\frac12}\le \frac{C_k}{\log\left(2+t\right)^k}\left\|\left(u_0, u_1\right)\right\|_{H^{k+1}\times H^k}.
\end{equation}
\end{theorem}
Though the idea of the proof is similar to that of \cite{bj16}, we need new tools because of we are leaving $\setr^d$. In \cite{bj16}, they used the fact that $\prod_{i=1}^d\cos(\pi x_i)\in C^\infty_b(\setr^d)$ has critical points exactly at $\mathbb{Z}^d\subset\setr^d$. On our $(M, g)$, neither the function nor the $\mathbb{Z}^d$-structure remains. We manage to get this fixed on cylindrical ends, but it remains unfixable on those sub-conic ends. 

To counter such difficulty in dealing with the subconic ends, we develop a novel Carleman estimate using a finite family of weight functions on manifolds of bounded geometry without boundary. The idea is based on that of two-weight Carleman estimates in bounded domain developed in \cite{bur98}. The new estimate allows us to construct on each end a finite family of Carleman weights, possibly very degenerate or even identically a constant somewhere, to cover the whole manifold and to give a global Carleman estimate. To our knowledge, this global Carleman estimate using finitely many weight functions has not been employed previously. This Carleman estimate with multiple weights might be interesting on its own for other applications. 

We note here that the regularity assumptions upon the damping $a$ in these two theorems can be weakened. In Theorem \ref{LT1} we only need $a\in L^\infty(M)$ to be uniformly continuous, and in Theorem \ref{LT2} we only need $a\in L^\infty(M)$. We choose not to develop those improvements here but they follow from the strategy described in \cite{bj16}. 

We organise our paper in the following order: in Section 2, we introduce our Carleman estimate with multiple weights; in Section 3, we show there exists a family of Carleman weight functions on our prespecified manifold $(M, g)$ compatible with the Carleman estimate developed in Section 2; in Section 4, we finish the proof of Theorem \ref{LT1} concluding the exponential decay; in Section 5, we finish the proof of Theorem \ref{LT2} concluding the logarithmic decay. An appendix on analysis on manifolds of bounded geometry is attached at the end of the paper. 

\subsection{Acknowledgement}The author is grateful to Jared Wunsch for numerous discussions around these results as well as many valuable comments on the manuscript, and Nicolas Burq for helpful discussion and pointing out the possibility to use the two-weight Carleman estimate, and Jeffrey Rauch and Jacob Shapiro for their insightful comments. The author is grateful to two anonymous referees for kindly reading this manuscript and providing many valuable remarks.


\section{Carleman estimates with multiple weights}
Let $M$ be a manifold of bounded geometry, without boundary. See Appendix \ref{SA} for further details. Let $\Omega\subset M$ be an open set. 
\begin{definition}[Compatibility conditions] We say a finite family of weight functions $\{\psi_1,\dots,\psi_n\}\subset C^\infty_b(M)$ is \emph{compatible with control from $\Omega$}, if there exists an open set $\Omega_0\subset \Omega$ with the following properties:
\begin{enumerate}
\item We have $d(\Omega_0, M\setminus \Omega)>0$ where $d$ is the distance on $M$.
\item There exist constants $\rho, \tau>0$ such that, at each point $x\in M\setminus \Omega_0$, for each $k$, if $\abs{\nabla_g \psi_k(x)}<2\rho$, then there exists some $l$ that
\begin{equation}\label{S2LL20}
\abs{\nabla_g \psi_l(x)}\ge2\rho, \quad \psi_l(x)\ge \psi_k(x)+\tau.
\end{equation}
\end{enumerate}
\end{definition}
It is natural to impose the compatibility condition upon the weight functions. We aim to control the $L^2$-size of a quasi-mode by merely the $L^2$-size of that inside the region of control $\Omega$. At $x$ outside the region of control, if we allow some weight functions to have vanishing gradients, they will not control the size of the quasi-mode locally near $x$. Therefore there has to be another weight whose gradient is sufficiently large to control that locally near $x$. This explains the first part of \eqref{S2LL20}. 

On another hand, at such a point $x$, because we use the exponential weights $\exp(e^{\lambda\psi_l}/h)$ whose control is exponentially weak, we do not want this very weak control to be cloaked by the large exponential sizes of other non-controlling weights $\exp(e^{\lambda\psi_k}/h)$. To avoid that, we ask for a fixed gap between non-controlling and controlling weights, as in the second part of \eqref{S2LL20}. Then we have $\exp(e^{\lambda\psi_k}/h)\le e^{-\epsilon/h}\exp(e^{\lambda\psi_l}/h)$ for some uniform constant $\epsilon>0$ depending on $\tau$. Now we note that the control induced by $\psi_l$ is observable, in the sense that the non-controlling weight $\psi_k$ generates an exponentially weaker term.  

\begin{theorem}[Global Carleman estimates with multiple weights]\label{S2T2}
For $M$ a manifold of bounded geometry without boundary, assume there are non-negative Carleman weights $\psi_1, \dots, \psi_n$ compatible with the control from $(\Omega, \Omega_0)$ in the sense of \eqref{S2LL20}. Then, we have a global Carleman estimate with constant $C>0$, independent of semiclassical parameter $h\in(0, h_0)$ for small $h_0$, such that
\begin{equation}
\left\|u\right\|_{L^2(M)}\le e^{C/h}\left(\left\|\left(h^2\Delta_g-V(x;h)\right) u\right\|_{L^2(M)}+\left\|u\right\|_{L^2(\Omega)}\right),
\end{equation}
where $V\in C^\infty_b(M\times[0, h_0])$ is a semiclassical uniformly bounded real potential. 
\end{theorem}
\begin{proof}
1. We start by deriving a local estimate via the hypoelliptic arguments. First note we can write
\begin{equation}
V(x;h)=V_0(x)+hV_1(x)+h^2 V_2(x;h),
\end{equation}
where $V_0,V_1\in C^\infty_b(M)$ and $V_2\in C^\infty_b(M\times[0, h_0])$. Fix a cutoff $\chi\in C^\infty_b(M)$ such that $\chi\equiv 1$ on $M\setminus \Omega$ and identically $0$ on $\Omega_0$. Fix a $\psi_k$ and denote $U_k^{\upsilon}=\{x\in M: \abs{\nabla_g \psi_k(x)}<\upsilon\}$. For each $k$, fix a $\chi_k\in C^\infty_b(M)$ such that $\chi_k\equiv1$ on $M\setminus U^{2\rho}_k$ and identically $0$ on $U^\rho_k$. Set $P_h=h^2\Delta_g-V(x;h)$.
Construct the exponential Carleman weights by $\phi_k=e^{\lambda \psi_k}$, where $\lambda$ is some large number to be determined later, and the conjugated operator by
\begin{multline}\label{S2L7}
P_{k, h}=e^{\phi_k/h} P_h e^{-\phi_k/h}=\left(h^2\Delta_g-\abs{\nabla_g\phi_k}^2-V_0(x)-h V_1(x)-h^2V_2(x;h)\right)\\
+2h\nabla^j \phi_k\nabla_j-h\Delta_g \phi_k. 
\end{multline}
See \eqref{A1L8} for the notation $\nabla^j \phi_k\nabla_j$. Note
\begin{multline}\label{S2LL28}
\left\|P_{k,h}u\right\|^{2}_{L^2}=\left\|P^*_{k,h}u\right\|^{2}_{L^2}+\left\langle [P_{k,h}^*, P_{k,h}]u, u\right\rangle\ge \left\langle [P_{k,h}^*, P_{k,h}]u, u\right\rangle\\
=h\left\langle \oph(i^{-1}\{\overline{p_{k,h}},p_{k,h}\})\right\rangle+h^2\langle R_2u, u\rangle
=2h\left\langle \oph(\{p_{k,h}^R,p_{k,h}^I\})\right\rangle+h^2\langle R_2u, u\rangle,
\end{multline}
where $R_2\in \Psi_{u,h}^2$ and the real and imaginary parts of the principal symbol are
\begin{equation}
p_k^R=\abs{\xi}^2-\abs{\nabla_g \phi_k}^2-V_0(x),\quad p_k^I=2\xi(\nabla_g\phi_k).
\end{equation} 
Denote the subset of the cotangent bundle that contains the characteristic set, 
\begin{multline}\label{S2LL10}
S_k=\\
\left\{(x,\xi)\in T^*\left(M\setminus(\Omega_{0}\cup U_k^\rho) \right): \frac14\left(\abs{\nabla_g \phi_k}^2+V_0(x)\right)\le \abs{\xi}^2\le 4\left(\abs{\nabla_g \phi_k}^2+V_0(x)\right)\right\},
\end{multline}
outside of which $(p_k^R)^2\ge 9/16$. Consider a microlocal cutoff $b_k(x,\xi)\in S_u^0(M)$ that is supported in $S_k$ and is identically 1 on $\{(x,\xi)\in T^*(M\setminus(\Omega_0\cup U^\rho_k)): 1/2\le (\abs{\nabla_g\phi_k}^2+V_0)^{-1}\abs{\xi}^2\le 2\}$. Note $\oph(1-b_k)$ and $\oph((1-b_k)p_{k,h}^{-1})P_{k,h}$ have the same principal symbol, so
\begin{equation}\label{S2LL26}
\left\|\oph(1-b_k)u\right\|_{H^2_h}\le C\left\|P_{k,h} u\right\|+Ch\left\|u\right\|_{H^1_h}. 
\end{equation}
On another hand, let $b_k'=b_k\langle \xi\rangle^2$ then $\oph(b_k')$ and $\langle hD\rangle^2\oph{b_k}$ are both in $\Psi_u^{-\infty}$ and their principal symbols agree. Thus
\begin{equation}\label{S2LL27}
\left\|\oph(b_k)u\right\|_{H^2_h}\le C\|\oph(b_k')u\|+Ch\|u\|_{L^2}.
\end{equation}
From \eqref{S2LL26} and \eqref{S2LL27} we know that 
\begin{equation}\label{S2LL29}
\left\|u\right\|_{H^2_h}\le C\left\|P_{k,h}u\right\|+C\|\oph(b_k')u\|.
\end{equation}

We claim that $\left\{p_k^R,p_k^I\right\}$ is uniformly bounded from below on $S_k$. Note that at any $(x,\xi)\in S_k\subset T^*(M\setminus U_k^\rho)$ we have $\abs{\nabla_g \psi_k}\ge \rho$ from the definition of $U_k^\rho$. Hence on $S_k$ we have $\abs{\nabla_g \phi_k}=\lambda\abs{\nabla_g \psi_k}e^{\lambda\psi_k}\ge \lambda\rho e^{\lambda\psi_k}$ and therefore
\begin{equation}
\frac14\left(\lambda^2\rho^2e^{2\lambda\psi_k}+V_0(x)\right)\le \abs{\xi}^2\le 4\left(\lambda^2C_ke^{2\lambda\psi_k}+V_0(x)\right),
\end{equation}
where $C_k$'s are some constants dependent only on the maximal size of first derivatives of $\psi_k$'s. 
Let $(x,\xi)\in S_k$. In the canonical coordinates induced by the geodesic normal coordinates around $x$, we have $g_{ij}(x)=\delta_{ij}$ and $\nabla g_{ij}(x)=0$. Compute, by noting that at $x$ we have $\nabla_g=\nabla$, the Euclidean gradient, 
\begin{multline}\label{S2LL11}
\pd_\xi p_{k}^R.\pd_x p_{k}^I=4\lambda^2e^{\lambda\psi_k}\abs{\nabla\psi_k}^2\abs{\xi}^2+4\lambda e^{\lambda\psi_k}\xi^t.(\nabla^2\psi_k).\xi\\
\ge 0+\bigo\left(\lambda^3e^{3\lambda\psi_{k}}\right)+\bigo\left(\lambda e^{\lambda\psi_{k}}\right)
\end{multline}
and
\begin{multline}
-\pd_x p_{k}^R.\pd_\xi p_{k}^I=2\lambda^4e^{3\lambda\psi_{k}}\abs{\nabla \psi_{k}}^4+2\lambda e^{\lambda\psi_k}\nabla\psi^t.\nabla V_0\\+\lambda^3 e^{3\lambda\psi_{k}}\left(\nabla\abs{\nabla\psi_{k}}^2\right)^t.\nabla\psi_{k}
\ge 2\rho^4\lambda^4e^{3\lambda\psi_{k}}+\bigo\left(\lambda e^{\lambda\psi_k}\right) +\bigo\left(\lambda^3 e^{3\lambda\psi_{k}}\right),
\end{multline}
from the uniform boundedness of the derivatives of order up to 2 of $\psi_k$. Hence on $S_k$ with a large $\lambda$, 
\begin{equation}
\left\{p_{k}^R, p_{k}^I\right\}\ge 2\rho^4\lambda^4e^{3\lambda\psi_{k}}+\bigo\left(\lambda^3 e^{3\lambda\psi_{k}}\right)\ge C\lambda^3 e^{3\lambda\psi_{k}}\ge 9/16>0.
\end{equation}

We conclude that throughout $T^*(M\setminus(\Omega_{2\beta}\cup U_k^\rho))$, we have 
\begin{equation}\label{S2LL12}
\eta (1-b_k)^2\left\langle \xi\right\rangle^3 +\left\{p_k^R,p_k^I\right\}\ge 9/16>0
\end{equation}
for some fixed $\eta$ large, and hence there exists $C>0$ such that 
\begin{equation}
\eta (1-b_k)^2\left\langle \xi\right\rangle^3 +\left\{p_k^R,p_k^I\right\}\ge C\langle \xi\rangle^3
\end{equation}
Now by invoking the weak Garding inequality in Proposition \ref{A1T1} on $M\setminus(\Omega_{0}\cup U_k^\rho)$ we have for any $u\in L^2(M)$ with support inside $M\setminus(\Omega_{2\beta}\cup U_k^\rho)$,
\begin{equation}
\left\langle \oph\left(\{p_k^R, p_k^I\}+\eta(1-b_k)^2\langle \xi\rangle^3\right)u, u\right\rangle\ge C\|u\|_{H^{3/2}_h}^2. 
\end{equation}
This implies
\begin{equation}
\left\langle \oph(\{p^R_k, p^I_k\})u, u\right\rangle\ge C\|u\|_{H^{3/2}_h}^2-C\|\oph(1-b_k) u\|_{H^{3/2}_h}^2-Ch\|u\|_{H^2_h}^2.
\end{equation}
From \eqref{S2LL28}, we use \eqref{S2LL26} and \eqref{S2LL29} to obtain
\begin{equation}
\left\|P_{k,h}u\right\|^2\ge Ch\|u\|_{H^{3/2}_h}^2-Ch\|P_{k,h} u\|^2
\end{equation}
and absorb the last term: for any $u$ with support inside $M\setminus \left(\Omega_{2\beta}\cup U^{\rho}_k\right)$ we have
\begin{equation}
\left\|P_{k,h}u\right\|\ge Ch^{\frac{1}{2}}\|u\|_{H^{3/2}_h}.
\end{equation}
Apply the above estimate to $e^{\phi_k/h}\chi_k\chi u$ to obtain the local estimate, 
\begin{equation}\label{S2L1}
\left\|e^{\phi_k/h}\chi_k\chi u\right\|_{L^2}\leq Ch^{-\frac12}\left\|P_{k, h}e^{\phi_k/h}\chi_k\chi u\right\|_{L^2}=Ch^{-\frac12}\left\|e^{\phi_k/h}P_{h}\chi_k\chi u\right\|_{L^2}
\end{equation}
from the claimed hypoellipticity.

2. We want to derive a crude version of the global estimate by just summing up the local estimates. Let \eqref{S2L1} be further simplified. Estimate
\begin{equation}\label{S2LL19}
\left\|e^{\phi_k/h}\chi_k\chi u\right\|_{L^2}\le Ch^{-\frac12}\left(\left\|e^{\phi_k/h}\chi_k\chi P_{h}u\right\|_{L^2}+\left\|e^{\phi_k/h}\left[P_{k, h},\chi_k\chi\right]  u\right\|_{L^2}\right)
\end{equation}
and
\begin{equation}\label{S2LL16}
\left[P_h, \chi_k\chi \right]u=h^2\nabla_g^*\left(\chi_k \nabla_g\chi+\chi\nabla_g\chi_k\right)u-2h^2 \chi_k \nabla^j\chi\nabla_j u-2h^2 \chi \nabla^j\chi_k\nabla_j u.
\end{equation}
Let $\kappa\in C^\infty_b(M)$ be a cutoff function supported inside $\Omega\setminus\Omega_0$ and being identically $1$ on the support of $\nabla_g \chi$, and let $\kappa_k\in C^\infty_b(M)$ be cutoff functions supported inside $U^{2\rho}_k\setminus U^{\rho}_k$ and identically $1$ on the support of $\nabla_g \chi_k$. We immediately have the estimates of the first two terms in \eqref{S2LL16}, 
\begin{equation}
\left\|h^2\nabla_g^*\left(\chi_k \nabla_g\chi+\chi\nabla_g\chi_k\right)u\right\|_{L^2} \le Ch^2 \left\|\kappa u\right\|_{L^2}+Ch^2 \left\|\kappa_k u\right\|_{L^2},
\end{equation}
from the uniform boundedness of first two derivatives of $\chi$ and $\chi_k$. Consider the next term, 
\begin{multline}\label{S2LL17}
\left\|2h^2\chi_k\nabla^j \chi \nabla_j u\right\|_{L^2}^2\le 4h^4\left\|\abs{\chi_k\nabla_g\chi} \abs{\nabla_g u}\right\|^2_{L^2}=4h^4\left\langle \abs{\chi_k\nabla_g\chi}^2\nabla_g u, \nabla_g u\right\rangle\\
=4h^4\left\langle \nabla^*_g \left(\abs{\chi_k\nabla_g\chi}^2\nabla_g u\right), u\right\rangle=-4h^4\cre\left\langle \nabla^j\abs{\chi_k\nabla_g\chi}^2\nabla_j u, u\right\rangle\\
+4h^4\cre\left\langle \abs{\chi_k\nabla_g\chi}^2\Delta_g u, u\right\rangle.
\end{multline}
The first term of the last line is estimated via an adjoint argument
\begin{equation}\label{S2LL18}
4h^4\cre\left\langle \left(\nabla^j\abs{\chi_k\nabla_g\chi}^2\right)\nabla_j u, u\right\rangle=2h^4\left\langle \left(\Delta_g\abs{\chi_k\nabla_g\chi}^2\right)u, u\right\rangle.
\end{equation}
Indeed, we have
\begin{multline}
\left\langle \left(\nabla^j\abs{\chi_k\nabla_g\chi}^2\right)\nabla_j u, u\right\rangle=\left\langle \nabla_g u, \left(\nabla_g\abs{\chi_k\nabla_g\chi}^2\right)u\right\rangle\\
=\left\langle  u, \nabla_g^*\left(\left(\nabla_g\abs{\chi_k\nabla_g\chi}^2\right)u\right)\right\rangle=\left\langle  u, \left(\left(\Delta_g\abs{\chi_k\nabla_g\chi}^2\right)u\right)\right\rangle-\left\langle  u, \left(\nabla^j\abs{\chi_k\nabla_g\chi}^2\right)\nabla_j u\right\rangle.
\end{multline}
We bring \eqref{S2LL17} and \eqref{S2LL18} together to see
\begin{multline}
\left\|2h^2\chi_k\nabla^j \chi \nabla_j u\right\|_{L^2}\le Ch^2\left\|\kappa u\right\|_{L^2}+Ch^2\left\|\kappa\Delta_g u\right\|_{L^2}=Ch^2\left\|\kappa u\right\|_{L^2}\\
+C\|\kappa P_h u\|_{L^2}+C\|\kappa V(x) u\|_{L^2}\le C\left\|\kappa u\right\|_{L^2}+C\|P_h u\|_{L^2},
\end{multline}
since the third order derivatives of $\chi$ are uniformly bounded and $V(x)$ is a bounded potential. Symmetrically we have
\begin{equation}
\left\|2h^2\chi\nabla^j \chi_k \nabla_j u\right\|_{L^2}\le C\left\|\kappa_k u\right\|_{L^2}+C\|P_h u\|_{L^2}
\end{equation}
and a complete estimate of \eqref{S2LL16},
\begin{equation}
\left\|\left[P_h, \chi_k\chi\right]u\right\|_{L^2}\le C\|P_h u\|_{L^2}+C\left\|\kappa u\right\|_{L^2}+C\left\|\kappa_k u\right\|_{L^2}.
\end{equation}
Finally we have for $h\in (0, h_0]$, 
\begin{multline}\label{S2L2}
\left\|e^{\phi_k/h}\chi_k\chi u\right\|_{L^2}\le Ch^{-\frac12}\left\|e^{\phi_k/h}P_h u\right\|_{L^2}+Ch^{-\frac12}\left\|e^{\phi_k/h}\kappa_k u\right\|_{L^2}\\
+Ch^{-\frac12}\left\|e^{\phi_k/h}\kappa u\right\|_{L^2}.
\end{multline}
Note that the constants $\lambda, C, h_0>0$ could be chosen uniformly such that \eqref{S2L2} holds for each $k$. Sum up \eqref{S2L2} over $k=1,\dots, n$ to get the crude version of the global estimate,
\begin{multline}\label{S2L4}
\sum_{k=1}^n\left\|e^{\phi_k/h}\chi_k\chi u\right\|_{L^2}\le Ch^{-\frac12}\sum_{k=1}^n\left\|e^{\phi_k/h}P_h u\right\|_{L^2}+Ch^{-\frac12}\sum_{k=1}^n\left\|e^{\phi_k/h}\kappa_k u\right\|_{L^2}\\+Ch^{-\frac12}\sum_{k=1}^n\left\|e^{\phi_k/h}\kappa u\right\|_{L^2},
\end{multline}
in which $\kappa$ is supported inside $\Omega\setminus\Omega_0$, and $\kappa_k$ is supported inside $U^{2\rho}_k\setminus U^\rho_k$.

3. Finally we use the compatibility condition \eqref{S2LL20} imposed upon the Carleman weights to refine the global estimate \eqref{S2L4}. Recall that $U_{l}^{2\rho}=\{x\in M: \abs{\nabla_g \psi_l}<2\rho\}$ stands for the points at which the weight $\psi_l$ fails to control the quasimode. Given the assumptions on compatibility \eqref{S2LL20}, at each $x\in U_l^{2\rho}$, there exists some $m$ such that
\begin{equation}\label{S2LL21}
\psi_{l}\le \psi_{m}-\tau, \quad \phi_m\ge \phi_l+\left(e^{\lambda\tau}-1\right)e^{\lambda\psi_l}\ge \phi_l+\epsilon,
\end{equation}
where
\begin{equation}
\epsilon=\left(e^{\lambda\tau}-1\right)e^{\lambda\min_{k=1}^n(\inf_{M}\psi_k)}>0.
\end{equation}
Note that $\epsilon$ does not depend on $l, m$. We have
\begin{equation}\label{S2LL22}
e^{\phi_{l}/h}\le e^{-\epsilon/h}e^{\phi_{m}/h}.
\end{equation}
Now at each $x\in  M\setminus \Omega_0$, we can partition $\{1,\dots, n\}$ into $\{l_1,\dots, l_{n-q}\}$ and $\{m_1,\dots, m_q\}$ with some $0<q\le n$, where
\begin{equation}
\abs{\nabla_g \phi_{l_*}(x)}<2\rho, \quad \abs{\nabla_g \phi_{m_*}(x)}\ge 2\rho.
\end{equation}
For each $i=1,\dots, n-q$, as $x\in U^{2\rho}_{l_i}$, there is some $m_j$ such that 
\begin{equation}
\phi_{l_i}\le \phi_{m_j}-\epsilon,\quad e^{\phi_{l_i}/h}\le e^{-\epsilon/h}e^{\phi_{m_j}/h}\le e^{-\epsilon/h}\sum_{j=1}^{q}e^{\phi_{m_j}/h},
\end{equation}
from \eqref{S2LL21} and \eqref{S2LL22}. Then 
\begin{equation}
\sum_{i=1}^{n-q}e^{\phi_{l_i}/h}\le \left(n-q\right)e^{-\epsilon/h}\sum_{j=1}^{q}e^{\phi_{m_j}/h}.
\end{equation}
As $x\notin U_{m_j}^{2\rho}$, the cutoffs $\chi_{m_j}$'s are all $1$, and therefore
\begin{multline}\label{S2L5}
\sum_{k=1}^n e^{\phi_k/h}\chi_k\ge\sum^q_{j=1}e^{\phi_{m_j}/h}\ge \frac12 \sum^q_{j=1}e^{\phi_{m_j}/h}+\frac{1}{2(n-q)}e^{\epsilon/h}\sum_{l_i}^{n-q}e^{\phi_{l_i}/h}\\
\ge \frac{1}{2n}\sum_{k=1}^n e^{\phi_k/h},
\end{multline}
for $0<q<n$. Note that when $q=n$, \eqref{S2L5} holds trivially because all $\chi_k$'s are 1, and therefore we conclude that \eqref{S2L5} holds at each $x\in M\setminus\Omega_0$. This improves the estimate \eqref{S2L4} from the left:
\begin{equation}\label{S2LL23}
\sum_{k=1}^n\left\|e^{\phi_k/h}\chi_k\chi u\right\|_{L^2}\ge\left\|\sum_{k=1}^ne^{\phi_k/h}\chi_k\chi u\right\|_{L^2}\ge \frac{1}{2n}\left\|\left(\sum_{k=1}^ne^{\phi_k/h}\right)\chi u\right\|_{L^2}
\end{equation}
as $\chi$ is supported inside $M\setminus \Omega_0$. 

Meanwhile, for each $l=1,\dots, n$, the cutoff function $\kappa_l$ is supported in $U_{l}^{2\rho}$. Therefore at each $x\in U_{l}^{2\rho}$, we know from \eqref{S2LL21} and \eqref{S2LL22} that
\begin{equation}\label{S2L6}
e^{\phi_l/h}\kappa_l\le e^{-\epsilon/h}e^{\phi_m/h}\le e^{-\epsilon/h}\sum_{k=1}^ne^{\phi_k/h}.
\end{equation}
This inequality outside $U_l^{2\rho}$ holds trivially as $\kappa_l$ vanishes, and hence holds everywhere in $M\setminus \Omega_0$. It improves the second term on the right in \eqref{S2L4}, 
\begin{multline}\label{S2LL24}
Ch^{-\frac12}\sum_{l=1}^n\left\|e^{\phi_l/h}\kappa_l u\right\|_{L^2}\le Ch^{-\frac12}\sum_{l=1}^n\left\|e^{-\epsilon/h}\sum_{k=1}^ne^{\phi_k/h} u\right\|_{L^2}\\
=Cn h^{-\frac12}e^{-\epsilon/h}\left\|\left(\sum_{k=1}^ne^{\phi_k/h}\right)u\right\|_{L^2}.
\end{multline}
Bring \eqref{S2LL23} and \eqref{S2LL24} into \eqref{S2L4} to observe
\begin{multline}
\frac1{2n}\left\|\left(\sum_{k=1}^n e^{\phi_k/h}\right)\chi u\right\|_{L^2}\le Ch^{-\frac12}\sum_{k=1}^n\left\|e^{\phi_k/h}P_h u\right\|_{L^2}\\
+Cnh^{-\frac12}e^{-\epsilon/h}\left\|\left(\sum_{k=1}^ne^{\phi_k/h}\right)u\right\|_{L^2}+Ch^{-\frac12}\sum_{k=1}^n\left\|e^{\phi_k/h}\kappa u\right\|_{L^2}.
\end{multline}

Finally, as $1-\chi$ and $\kappa$ are both supported inside $\Omega$ and are uniformly bounded, we can bound the $L^2(M)$-norm of terms $e^{\phi_k/h} (1-\chi)u$ and $e^{\phi_k/h} \kappa u$ by the $L^2(\Omega)$-norm of $e^{\phi_k/h} u$, that is, 
\begin{multline}\label{S2LL14}
\left\|\left(\sum_{k=1}^n e^{\phi_k/h}\right)u\right\|_{L^2(M)}\le \left\|\left(\sum_{k=1}^n e^{\phi_k/h}\right)\left(1-\chi\right) u\right\|_{L^2(M)}\\
+\left\|\left(\sum_{k=1}^n e^{\phi_k/h}\right)\chi u\right\|_{L^2(M)} \le Ch^{-\frac12}\sum_{k=1}^n\left\|e^{\phi_k/h}P_h u\right\|_{L^2(M)}\\
+Cnh^{-\frac12}e^{-\epsilon/h}\left\|\left(\sum_{k=1}^ne^{\phi_k/h}\right)u\right\|_{L^2(M)}+Ch^{-\frac12}\sum_{k=1}^n\left\|e^{\phi_k/h} u\right\|_{L^2(\Omega)}
\end{multline}
As $h^{-\frac12}e^{-\epsilon/h}\rightarrow 0$ semiclassically, we can absorb the second term on the right by the term on the left for small $h$, that is, 
\begin{multline}\label{S2LL25}
\left\|\left(\sum_{k=1}^n e^{\phi_k/h}\right)u\right\|_{L^2(M)}\le Ch^{-\frac12}\sum_{k=1}^n\left\|e^{\phi_k/h}P_h u\right\|_{L^2(M)}\\
+Ch^{-\frac12}\sum_{k=1}^n\left\|e^{\phi_k/h} u\right\|_{L^2(\Omega)}.
\end{multline}
Denote the global maximum and minimum over all $\phi_k$'s by 
\begin{equation}
K_+=\max_{1\le k\le n}\left(\sup_{x\in M}\phi_k(x)\right), \quad K_-=\min_{1\le k\le n}\left(\inf_{x\in M}\phi_k(x)\right),
\end{equation}
where $K_+>K_-$. Then we have from \eqref{S2LL25}, 
\begin{equation}
e^{{K_-}/h}\left\|u\right\|_{L^2(M)}\le Cnh^{-\frac12}e^{{K_+}/h}\left\|P_h u\right\|_{L^2(M)}+Cnh^{-\frac12}e^{{K_+}/h}\left\|u\right\|_{L^2(\Omega)},
\end{equation}
which is reduced to
\begin{equation}
\left\|u\right\|_{L^2(M)}\le e^{C/h}\left(\left\|P_h u\right\|_{L^2(M)}+\left\|u\right\|_{L^2(\Omega)}\right).
\end{equation}
This is our claim.
\end{proof}
\begin{remark}
\begin{enumerate}[wide]\label{S2T1}
\item When $M$ is a compact manifold without boundary, then one could control from any open set $\Omega$ with arbitrary open subset $\Omega_0$ with only one weight. It suffices to find a Morse function and find a diffeomorphism moving all critical points into $\Omega_0$. 
\item It is observed that the uniform gap $\psi_l\ge\psi_k+\tau$ in the compatibility condition is necessary. By this fixed gap $\tau$, we extracted an $e^{-\varepsilon/h}$-decay in \eqref{S2L6}, further leading to the absorption argument between \eqref{S2LL14} and \eqref{S2LL25}. Without such this uniform gap we see the inequality \eqref{S2LL14} will not generate any effective bound on $L^2(M)$-norm of $\sum_{k=1}^n e^{\phi_k/h}u$, as this term on the right is now of size $h^{-\frac12}\rightarrow \infty$ semiclassically.
\end{enumerate}
\end{remark}

\section{Construction of Carleman weights}
In this section, we aim to explicitly construct the weight functions on our prespecified manifold $(M, g)$ in \eqref{S1LL10} to obtain the global Carleman estimate we developed in the previous section. Assume throughout this section that the Network Control Condition $(L,\omega,2\beta, \{x_m\})$ defined in Definition \ref{S1T1} holds on $(M,g)$. Let $\Omega_{\upsilon}=\{x\in M: a(x)>\upsilon\}$. Our ultimate target in this section is to control the whole manifold $(M,g)$ from $\left(\Omega_\beta, \Omega_{2\beta}\right)$. 

The strategy is to start by working on the model manifold $(M, g_0)$ in \eqref{S1L9}. We construct a family of weights on each end, and another weight on the central compactum, then show they are compatible with control from $(\Omega_\beta, \Omega_{2\beta})$. Eventually we pull back the weights via $\Phi_0$ back to the prespecified $(M, g)$. 

Note that $\Phi_0(x)=x$ for each $x\in M$, and $\Phi_0(\Omega_{\upsilon})=\Omega_{\upsilon}$. We claim that a Network Control Condition $(L,\omega, 2\beta, \{x_m\})$ on $(M, g)$ implies another $(L',\omega', 2\beta, \{x_m\})$ on $(M,g_0)$. Recall that $d\Phi_0^{-1}$ is bounded from above and below, 
\begin{equation}\label{S3L8}
C_0\le \left\|d \Phi_0^{-1}\right\|\le C_1.
\end{equation}
and therefore
\begin{equation}\label{S3L7}
C_1^{-1}d_{g}(x,y)\le d_{g_0}(\Phi_0(x), \Phi_0(y))\le C_0^{-1}d_{g}(x,y)
\end{equation}
for each $x, y\in M$. Let $L'=C_0^{-1}L$ and $\omega'=C_0^{-1}\omega$. At each $x\in M$, we have
\begin{equation}
d_{g_0}(x, \bigcup_m \left\{x_m\right\})\le C_0^{-1}d_{g}(x,\bigcup_m \left\{x_m\right\})\le C_0^{-1}L=L',
\end{equation}
the last inequality of which comes from the Network Control Condition $(L,\omega, 2\beta, \{x_m\})$ on $(M,g)$. We also have $a(x)\ge 2\beta>0$ on $\bigcup_m B_{g_0}(x_m,\omega')\subset\bigcup_m B_{g}(x_m,\omega)$ as an immediate result of \eqref{S3L7}. Therefore we could, without loss of generality, assume a Network Control Condition $(L,\omega, 2\beta, \{x_m\})$ on $(M, g_0)$. Note that 
\begin{equation}\label{S3L9}
\bigcup_m B_{g_0}(x_m,\omega) \subset \Omega_{2\beta}. 
\end{equation}

We begin by constructing weight functions on the cylindrical ends, where the scaling functions $\theta_k$'s are identically 1. 

\begin{lemma}[Cylindrical ends]\label{S3T1}
Consider a cylindrical end $(M_k, g_0)$, that is $\pd M_k\times(1,\infty)_r$ endowed with the metric $g_0=dr^2+h$, where $h$ is a smooth metric on closed $\pd M_k$. There exists $\psi \in C^\infty_b(M_k)$, where $1\le\psi\le 3$ and there is some $\rho>0$ such that,
\begin{equation}\label{S3L1}
x\notin\Omega_{2\beta}\Rightarrow \abs{\nabla_{g_0} \psi(x)}\ge 2\rho.
\end{equation}
\end{lemma}
\begin{proof}
In this lemma we Let $M=M_k$ for we here only care what is happening on $M_k$.

1. We start by constructing a prototype weight based on the periodic structure. On $\pd{M}$ pick a Morse function $\psi_0\in C^\infty(\pd M)$ that is positive on $\pd M$. As $\pd M$ is compact, $\psi_0$ has $N$ critical points at $p_1,\dots, p_N\in \pd M$. Fix $\epsilon>0$ small. Let a periodic function $\psi_1(r)\in C^\infty(\left[1, \infty\right))$ be given by
\begin{equation}
\psi_1(r)=\cos\left(\frac{\pi\left(r-\left(1+2\epsilon\right)\right)}{2\left(L+4\omega\right)\left(N+1\right)}\right)+2,
\end{equation}
for small $\epsilon>0$. This is a function with a period of $4\left(L+4\omega\right)\left(N+1\right)$. Consider 
\begin{equation}\label{S3L2}
\tilde\psi_2(y, r)=\psi_0(y)\psi_1(r),
\end{equation}
and modify its size to get
\begin{equation}
\psi_2=1+2\left(\max{\tilde\psi_2}-\min{\tilde\psi_2}\right)^{-1}\left(\tilde\psi_2- \min{\tilde\psi_2}\right),
\end{equation}
where we note that $1\le \psi_2\le 3$ and we will later modify $\psi_2$ to move around the critical points. The critical points of $\psi_2$ are
\begin{equation}
p_{k, t}=\left(p_k, 1+2\epsilon+2 t \left(L+4r\right)\left(N+1\right)\right), \quad t\in \mathbb{N}_0, \quad k=1, \dots, N,
\end{equation}
where $\mathbb{N}_0=\mathbb{N}\cup \{0\}$. 

\begin{figure}
\begin{tikzpicture}[minimum size=0.01cm]
\draw (0,0) ellipse (0.5 and 1.25) node {$\partial M$};
\draw (0, -1.25) -- (5, -1.25);

\draw [dashed] (1, 1.25) arc (90:270:-0.5 and 1.25);
\draw [dashed] (2, 1.25) arc (90:270:-0.5 and 1.25);
\draw [dashed] (3, 1.25) arc (90:270:-0.5 and 1.25);
\draw [dashed] (4, 1.25) arc (90:270:-0.5 and 1.25);

\draw (5, 1.25) arc (90:270:-0.5 and 1.25);
\draw (0,1.25) -- (5,1.25);  

\draw [dashed] (-0.5,0) ellipse (0.5 and 1.25);
\draw (-0.5, -1.25) -- (0, -1.25);
\draw (-0.5, 1.25) -- (0, 1.25);

\draw [dashed] (5.5, 1.25) arc (90:270:-0.5 and 1.25);
\draw (5, -1.25) -- (5.5, -1.25);
\draw (5, 1.25) -- (5.5, 1.25);

\draw [->] (0.35, 0.89) node [circle, minimum size=1mm, inner sep=0, draw=black, fill=white,label={[label distance=-0.13cm]left:{$p_{1,0}$}}] {} -- (1.35, 0.89) node [circle, minimum size=1mm, inner sep=0,fill=black, label={[label distance=0cm]right:{$p_{1,0}'$}}] {};
\draw [->] (0.48, 0.35) node [circle, minimum size=1mm, inner sep=0, draw=black, fill=white,label={[label distance=-0.13cm]left:{$p_{2,0}$}}] {} -- (2.48, 0.35) node [circle, minimum size=1mm, inner sep=0,fill=black, label={[label distance=0cm]right:{$p_{2,0}'$}}] {};
\draw [->] (0.48, -0.35) node [circle, minimum size=1mm, inner sep=0, draw=black, fill=white,label={[label distance=-0.13cm]left:{$p_{3,0}$}}] {} -- (3.48, -0.35) node [circle, minimum size=1mm, inner sep=0,fill=black, label={[label distance=0cm]right:{$p_{3,0}'$}}] {};
\draw [->] (0.35, -0.89) node [circle, minimum size=1mm, inner sep=0, draw=black, fill=white,label={[label distance=-0.13cm]left:{$p_{4,0}$}}] {} -- (4.35, -0.89) node [circle, minimum size=1mm, inner sep=0,fill=black, label={[label distance=0cm]right:{$p_{4,0}'$}}] {};
\end{tikzpicture}
\caption{$\Phi_2$ stretches the critical points of $\psi_2$ apart.}
\label{F3}
\end{figure}
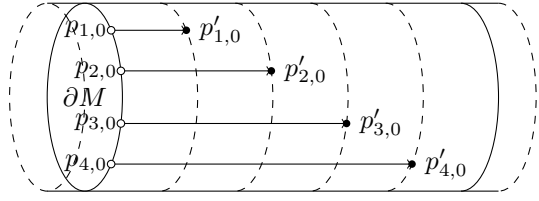

2. We modify the weight to have critical points of distance uniformly bounded from below by $2(L+4\omega)$ from each other. For $1\le k \le N$, define the flows $\gamma_k(s)$ for $s\in[0,\infty)$ by
\begin{equation}
\gamma_k(s): \left(y, r\right)\mapsto \left(y, r+2k\left(L+4\omega\right)s\right),
\end{equation}
generated by the constant radial vector fields $2k\left(L+4\omega\right)\pd_r$. Also denote by the flow $\gamma(t)$, 
\begin{equation}
\gamma(t): \left(y, r\right)\mapsto \left(y, r+4(N+1)\left(L+4\omega\right)t\right),\quad t\in \mathbb{N}_0
\end{equation}
that preserves the periodicity of $\psi_2$, in the sense that $\gamma(t)^*\psi_2=\psi_2$ for each $t\in\mathbb{N}_0$. Note that the flow $\gamma_k(1)$ pulls back points
\begin{equation}
p'_{k,t}=\left(p_k, 1+2\epsilon+2 t \left(L+4r\right)\left(N+1\right)+2k\left(L+4\omega\right)\right), \quad t\in\mathbb{N}_0
\end{equation}
to critical points $p_{k,t}$, that is, $\gamma_k(1): p_{k,t}\mapsto p'_{k,t}$. Let $\Gamma_{k,t}=\gamma_{k}\left([0,1]\right)(p_{k,t})$ and $\Gamma_{k,t}^{\upsilon}$ to be the $\upsilon$-neighbourhood of $\Gamma_{k, t}$, for $\upsilon=\epsilon,\epsilon/2$. Now for $t=0,1$, construct diffeomorphisms $\phi_{k, t}\in C^\infty(\Gamma_{k,t}^\epsilon; \Gamma_{k,t}^\epsilon)$ with inverses $\phi^{-1}_{k, t}\in C^\infty(\Gamma_{k,t}^\epsilon; \Gamma_{k,t}^\epsilon)$ such that
\begin{equation}
\begin{cases}
\phi_{k,t}: p'_{k, t}=\gamma_{k}(1)p_{k, t}\mapsto p_{k, t}\\
\phi_{k, t}=\id,~\textrm{on}~{\Gamma_{k, t}^\epsilon\setminus \Gamma_{k, t}^{\epsilon/2}}
\end{cases}.
\end{equation}
Because that all $\Gamma_{k,t}^\epsilon$'s are disjoint and $\phi_{k, t}=\id$ on $\Gamma_{k, t}^\epsilon\setminus \Gamma_{k, t}^{\epsilon/2}$, we can glue up $\phi_{k,t}$'s to obtain a diffeomorphism $\Phi_1$ on $\pd M\times\left[1, 1+4\left(L+4r\right)\left(N+1\right)\right]$. 
Note that $\Phi_1, \Phi_1^{-1}$ have all derivatives uniformly bounded from above and below, as the domain is compact. Also note that $\Phi_1$ is the identity on $\pd M\times[1, 1+\epsilon]$ and $\pd M\times[1+4\left(L+4r\right)\left(N+1\right)-\epsilon, 1+4\left(L+4r\right)\left(N+1\right)]$. This enables us to extend $\Phi_1$ periodically to some $\Phi_2$ on $M$, by defining on $\pd M\times[1+4t(L+4r)(N+1), 1+4(t+1)4t(L+4r)(N+1)]$, for each $t\in \mathbb{N}_0$,
\begin{equation}
\Phi_2=\gamma^{-1}(t)^*\Phi_1\gamma(t)^*.
\end{equation}
Let $\psi_3=\Phi_2^*\psi_2$, whose critical points are
\begin{equation}
p'_{k, t}=\left(p_k, 1+2\epsilon+2t\left(L+4\omega\right)\left(N+1\right)+2k\left(L+4\omega\right)\right).
\end{equation}
for each $t\in\mathbb{N}_0$, each $k=1,\dots, N$. Renumber those critical points by $p'_m$. We remark that any two critical points of ${\psi_3}$ are separated by distance of at least $2\left(L+4\omega\right)$. We also note for each $R>0$ one has $\abs{\nabla_{g_0}\psi_3}\ge C$ outside $\bigcup_m B(p'_m, R)$ for constant $C>0$ only depending on $R$, because $\psi_3$ is still periodic.

3. Finally we modify the weight function in uniform radius balls around critical points to obtain \eqref{S3L1}. Note that the balls
\begin{equation}\label{S3L3}
\bar{B}\left(p'_{m}, L+3\omega\right) \cap \bar{B}\left(p'_{m'}, L+3\omega\right)=\emptyset
\end{equation}
for any pair of critical points $p'_{m}$ and $p'_{m'}$. By the Network Control Condition, in each ball $B(p'_{m}, L+2\omega)$ we can find some $x_{m}$ in the network such that
\begin{equation}
\bar{B}\left(x_{m}, \omega\right)\subset B\left(p'_{m}, L+2\omega\right),
\end{equation}
and $a\ge 2\beta$ on $\bar{B}(x_{m}, \omega)$. Now in each ball $B(p'_{m}, L+3\omega)$, find a diffeomorphism $\phi'_{m}$ such that
\begin{equation}\label{S3L4}
\begin{cases}
\phi_{m}': x_m\mapsto p'_m\\
\phi_{m}'=\id,~\textrm{on}~B\left(p'_m, L+3\omega\right)\setminus B\left(p'_{m}, L+2\omega\right)
\end{cases}.
\end{equation}
Glue up $\phi'_{m}$'s to get a diffeomorphism $\Phi_3$ on $M$. We remark here that we can make this construction uniform in the sense that both $\Phi_3, \Phi_3^{-1}$ are in $C_b^{\infty}(M)$, as in \cite{bj16,rm16}. Therefore we have $B(p'_{m}, R)=B(\Phi_3\left( x_{m}\right), R)\subset \Phi_3\left(B\left(x_{m}, \omega\right)\right)$ for some $R>0$ uniform in all $m$. Now set $\psi=\Phi^*_3{\psi_3}$. We know $\abs{\nabla_{g_0}{\psi_3}}\ge  C>0$ uniformly for all $p\notin \cup_{m}B(p'_{m}, R)$. Hence for any $x\notin \bigcup_m B(x_m, \omega)$, we have $\Phi_3(x)\notin \bigcup_m B(p_m', R)$, and again by the boundedness of $d\Phi_3^{-1}$ we have $\abs{\nabla_{g_0} \psi}\ge 2\rho>0$ for some uniform $\rho$. As in \eqref{S3L9}, $\bigcup_m B(x_m, \omega)\subset \Omega_{2\beta}$, the claim holds. 

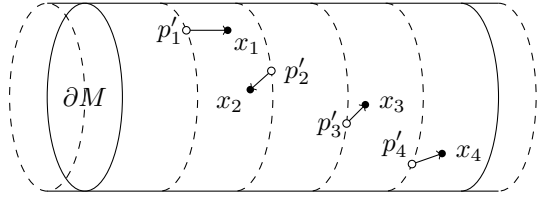
\begin{figure}
\begin{tikzpicture}[minimum size=0.01cm]
\draw (0,0) ellipse (0.5 and 1.25) node {$\partial M$};
\draw (0, -1.25) -- (5, -1.25);

\draw [dashed] (1, 1.25) arc (90:270:-0.5 and 1.25);
\draw [dashed] (2, 1.25) arc (90:270:-0.5 and 1.25);
\draw [dashed] (3, 1.25) arc (90:270:-0.5 and 1.25);
\draw [dashed] (4, 1.25) arc (90:270:-0.5 and 1.25);

\draw (5, 1.25) arc (90:270:-0.5 and 1.25);
\draw (0,1.25) -- (5,1.25);  

\draw [dashed] (-0.5,0) ellipse (0.5 and 1.25);
\draw (-0.5, -1.25) -- (0, -1.25);
\draw (-0.5, 1.25) -- (0, 1.25);

\draw [dashed] (5.5, 1.25) arc (90:270:-0.5 and 1.25);
\draw (5, -1.25) -- (5.5, -1.25);
\draw (5, 1.25) -- (5.5, 1.25);

\draw [->] (1.35, 0.89) node [circle, minimum size=1mm, inner sep=0, draw=black, fill=white,label={[label distance=-0.13cm]left:{$p_{1}'$}}] {} -- (1.9, 0.89) node [circle, minimum size=1mm, inner sep=0,fill=black, label={[label distance=-0.1cm]345:{$x_1$}}] {};
\draw [->] (2.48, 0.35) node [circle, minimum size=1mm, inner sep=0, draw=black, fill=white,label={[label distance=0cm]right:{$p_{2}'$}}] {} -- (2.2, 0.1) node [circle, minimum size=1mm, inner sep=0,fill=black, label={[label distance=-0.1cm]225:{$x_2$}}] {};
\draw [->] (3.48, -0.35) node [circle, minimum size=1mm, inner sep=0, draw=black, fill=white,label={[label distance=-0.13cm]left:{$p_3'$}}] {} -- (3.73, -0.1) node [circle, minimum size=1mm, inner sep=0,fill=black, label={[label distance=0cm]right:{$x_3$}}] {};
\draw [->] (4.35, -0.89) node [circle, minimum size=1mm, inner sep=0, draw=black, fill=white,label={[label distance=-0.13cm]155:{$p_4'$}}] {} -- (4.75, -0.75) node [circle, minimum size=1mm, inner sep=0,fill=black, label={[label distance=0cm]right:{$x_4$}}] {};
\end{tikzpicture}
\caption{$\Phi_3$ moves the critical points of $\psi_2$ into the sufficiently damped balls.}
\label{F4}
\end{figure}
\end{proof}
\begin{remark}\label{S3T2}
The radial stretch $\Phi_2$ is necessary here to pull the critical points sufficiently apart. Otherwise the balls $\bar B(p'_m, L+3\omega)$'s in \eqref{S3L3} may not be disjoint and the construction of the diffeomorphism $\Phi_3$ fails in \eqref{S3L4}.
\end{remark}
What makes this construction above interesting is that it is global on each cylindrical end, similar to the flavour of that on $\setr^d$ in \cite{bj16}. So it only takes a single weight function to control the whole end. However, it still relies much on the homogeneity of the space along the radial direction. Once we allow the scaling functions $\theta_k$'s to grow as $r\rightarrow\infty$, for example, on conic ends, this construction stops working, technically because there is no ideal way of constructing a product-type $\tilde\psi_2$ in \eqref{S3L2}. This constraint on subconic ends is removed by introducing a finite collection of weights.

\begin{lemma}[Subconic ends]\label{S3T3}
Consider a sub-conic end $(M_k, g_0)$, that is $M_k=\pd M_k\times(1,\infty)_r$ endowed with the metric $g_0=dr^2+\theta^2_k(r)h$, where $h$ is a smooth metric on closed $\pd M_k$, and $\theta_k(r)$ as described in \eqref{S1L1}. There exists some $R\ge1$ and let $M_R=\pd M_k\times(R,\infty)_r$, and for some finite $n\ge 1$ there exist $\psi_1,\dots, \psi_n \in C^\infty_b(M_R, \setr)$ that each $0\le \psi_k\le 3$, with a constant $\rho>0$, such that for all $k$, at each point $x\in M_R\setminus \Omega_{2\beta}$ with $\abs{\nabla_{g_0}\psi_k(x)}<2\rho$, there is some $l$ depending on $x$, such that $\psi_l(x)\ge 1$ and
\begin{equation}
\abs{\nabla_{g_0}\psi_l(x)}\ge2\rho,\quad \psi_l(x)\ge\psi_k(x)+1/2.
\end{equation}
\end{lemma}
\begin{proof}
In this lemma we write $M=M_k$ and $\theta=\theta_k$ for we here only care what is happening on $M_k$.
\begin{figure}
\begin{tikzpicture}[minimum size=0.01cm]
\draw (0,0) ellipse (0.125 and 0.625);
\draw (0, -0.625) -- (1, -1.25);

\draw (1, 1.25) arc (90:270:-0.25 and 1.25);
\draw (0,0.625) -- (1,1.25);  

\draw (0.096, 0.4) -- (1.192, 0.8);
\draw (0.096, -0.4) -- (1.192, -0.8);

\draw (3, -0.3125) -- (3, 0.3125);
\draw (5, -1.25) -- (5, 1.25);
\draw (3, -0.3125) -- (5, -1.25);
\draw (3, 0.3125) -- (5, 1.25);

\draw (3.25,0) node [circle, minimum size=0.5mm, inner sep=0,fill=black] {};
\draw (3.75,0) node [circle, minimum size=0.5mm, inner sep=0,fill=black] {};
\draw (4.25,0) node [circle, minimum size=0.5mm, inner sep=0,fill=black] {};
\draw (4.75,0) node [circle, minimum size=0.5mm, inner sep=0,fill=black] {};

\draw (3.75,-0.5) node [circle, minimum size=0.5mm, inner sep=0,fill=black] {};
\draw (3.75,0.5) node [circle, minimum size=0.5mm, inner sep=0,fill=black] {};
\draw (4.25,-0.5) node [circle, minimum size=0.5mm, inner sep=0,fill=black] {};
\draw (4.25,0.5) node [circle, minimum size=0.5mm, inner sep=0,fill=black] {};
\draw (4.75,-0.5) node [circle, minimum size=0.5mm, inner sep=0,fill=black] {};
\draw (4.75,0.5) node [circle, minimum size=0.5mm, inner sep=0,fill=black] {};
\draw (4.75,-1) node [circle, minimum size=0.5mm, inner sep=0,fill=black] {};
\draw (4.75,1) node [circle, minimum size=0.5mm, inner sep=0,fill=black] {};

\draw (0.266,0) node [circle, minimum size=0.5mm, inner sep=0,fill=black] {};
\draw (0.547,0) node [circle, minimum size=0.5mm, inner sep=0,fill=black] {};
\draw (0.828,0) node [circle, minimum size=0.5mm, inner sep=0,fill=black] {};
\draw (1.110,0) node [circle, minimum size=0.5mm, inner sep=0,fill=black] {};

\draw (0.526,0.419) node [circle, minimum size=0.5mm, inner sep=0,fill=black] {};
\draw (0.815,0.363) node [circle, minimum size=0.5mm, inner sep=0,fill=black] {};
\draw (1.100,0.332) node [circle, minimum size=0.5mm, inner sep=0,fill=black] {};
\draw (1.068,0.663) node [circle, minimum size=0.5mm, inner sep=0,fill=black] {};

\draw (0.526,-0.419) node [circle, minimum size=0.5mm, inner sep=0,fill=black] {};
\draw (0.815,-0.363) node [circle, minimum size=0.5mm, inner sep=0,fill=black] {};
\draw (1.100,-0.332) node [circle, minimum size=0.5mm, inner sep=0,fill=black] {};
\draw (1.068,-0.663) node [circle, minimum size=0.5mm, inner sep=0,fill=black] {};

\draw [->] (1.5,0) -- (2.7,0);
\draw (3.5, -1.5) node [label={{$E$}}] {};
\draw (0.3, -1.5) node [label={{$M$}}] {};
\draw (2.1,-0.2) node [label={{$\Phi\circ\Phi_k$}}] {};
\end{tikzpicture}
\caption{Placement of the critical points pulled back to $M$.}
\label{F5}
\end{figure}

1. We start by quasi-isometrically reducing the underlying geometry to an unbounded subset of $\setr^d$. As $\pd M$ is compact, it possesses a finite cover
\begin{equation}
\pd M\subset\bigcup_{k=1}^n (\phi^0_k)^{-1}(B(0,1)),
\end{equation}
such that each $\phi^0_k: \pd M \supset (\phi^0_k)^{-1}(B\left(0,1\right))\rightarrow B\left(0,1\right)\subset\setr^{d-1}$ is a $C_b^\infty$-diffeomorphism. For convenience, denote for $0\le \upsilon\le 1$, 
\begin{equation}
\Sigma_{\upsilon}=B(0,1- \upsilon), \quad \Sigma_{\upsilon}^c=B(0,1)\setminus B(0,1- \upsilon),
\end{equation}
where $\Sigma_{\upsilon}$ consists of points of distance more than $\upsilon$ away from the unit sphere, $\Sigma_{\upsilon}^c$ is the $\upsilon$-neighbourhood of the unit sphere, and $\Sigma_0$ is the unit open ball. Here, as the covers are open, one can fix a small $\epsilon >0$ such that for each $k$, for each $x$ such that $\phi^0_k(x)\in \Sigma^c_{4\epsilon}$, there is some $l$ such that $x\in (\phi^0_l)^{-1}\left(\Sigma_{4\epsilon}\right)$. Denote the model space by $(D,g_D)$ where
\begin{equation}
D=\Sigma_0\times(1,\infty)_r, \quad g_D=\theta^2(r)dy^2+dr^2.
\end{equation}
Construct diffeomorphisms $\Phi_k=\phi^0_k \otimes\id_r$ and observe that $\{\Phi_k^{-1}(D)\}$ covers $\pd M\times (1,\infty)$. Each map $\Phi_k$ is uniformly quasi-isometric with constants $C_1^+,C_1^->0$ such that
\begin{equation}
C_1^-d_{M}\left(x, x'\right)\le d_{D}\left(\Phi_k\left(x\right), \Phi_k\left(x'\right)\right)\le C_1^+d_{M}\left(x, x'\right),
\end{equation}
for any $x,x'\in M$. Consider $\Phi(y, r)=(\theta(r) y, r)$ for
\begin{equation}
\Phi: \left(D, g_D\right) \rightarrow \left(E, g_{\setr^d}\right), \quad E=\left\{(z',z_d=r): r\ge 1, \abs{z'}\le \theta(z_d)\right\}\subset \setr^d.
\end{equation}
This is a quasi-isometric $C_b^\infty$-diffeomorphism, that for any $(y, r), (y', r')\in D$ we have
\begin{equation}
C_2^-d_{D}\left((y, r), (y', r')\right)\le d_{E}\left(\Phi \left((y, r)\right),\Phi \left((y', r')\right)\right)\le C_2^+d_{D}\left((y, r), (y', r')\right).
\end{equation}
To verify the $C_b^\infty$ nature of $\Phi$, it suffices to first pull back $(D, g_D)$ to $(E, g_D')$, and then verify that the Christoffel symbols on $(E, g_D')$ are $C_b^\infty$ on $E$ as a subset of $\setr^d$. We omit the trivial computation here. Note that each $\Phi\circ\Phi_k$ is a $C^\infty_b$-diffeomorphism from $\phi_k^{-1}(D)$ to $E$ quasi-isometric with constants $C^\pm=C^\pm_1 C^\pm_2$. 

2. Now we construct on each $E\subset\setr^d$, a weight function with vanishing gradients exactly inside the damping balls given in the Network Control Condition. On $\setr^d$ we construct for $k=1,\dots, n$, 
\begin{equation}
\psi_k^0(z', z_d)=\cos\left(\frac{\pi z_d-2\pi \left(L+4\omega\right)kC^+}{2n\left(L+4\omega\right)C^+}\right)\prod_{j=1}^{d-1}\cos{\left(\frac{\pi z_j}{2\left(L+4\omega\right)C^+}\right)}+2,
\end{equation}
whose critical points are
\begin{equation}
p_m=\left(2C^+(L+4\omega)m', 2 C^+ (L+4\omega)(n m_d+k)\right)
\end{equation}
for all $m=(m', m_d)\in \mathbb{Z}^d$. See Figure \ref{F5}. Note that any two such critical points are of distance at least $2C^+(L+4\omega)$, measured in $\setr^d$. Set $R_0\ge1$ be the smallest constant such that for all $r\ge R_0$, we have $\theta(r)> C_1^- (L+4\omega)/\epsilon$. This lower bound on the radius guarantees that any point in $\Phi_k^{-1}(\Sigma_\epsilon\times(R_0,\infty))$ is of distance larger than $L+4\omega$ from the cross-sectional boundary $\Phi_k^{-1}(\pd\Sigma_0\times (R_0, \infty))$, measured in $M$. By the Network Control Condition, for all critical points $p_m$'s that are inside $\Phi(\Sigma_\epsilon\times(R_0,\infty))\subset E$ there exists at least a $x_m$ such that $B(x_m, \omega)\subset B((\Phi\circ\Phi_k)^{-1}(p_m), \left(L+2\omega\right))$ and $a\ge 2\beta$ on $B(x_m, \omega)$. Here
\begin{multline}
\Phi\circ\Phi_k\left(B(x_m, \omega)\right)\subset\Phi\circ\Phi_k\left(B((\Phi\circ\Phi_k)^{-1}(p_m), \left(L+2\omega\right))\right)\\
\subset B(p_m, C^+ \left(L+2\omega\right))
\end{multline}
are disjoint balls around $p_m$'s of some uniform radius. Hence via a process similar to the construction of the diffeomorphism $\Phi_3$ in the proof of \ref{S3T1}, we can find a $C^\infty_b$-diffeomorphism $\tilde\Phi_k$ on $E$, equal to the identity on $E\setminus\bigcup_m B(p_m, C^+(L+2\omega))$, such that $\tilde\Phi_k: \Phi_k(x_m)\mapsto p_m$. Set $\psi_k^1=\tilde\Phi_k^*\psi_k^0$, whose critical points in $\Phi(\Sigma_\epsilon\times (R_0,\infty))$ are a subset of $\{\Phi_k (x_m)\}$. Set $\psi_k^2=\Phi^*\psi_k^1$. This is a function defined on $\Sigma_0\times(R_0,\infty)\subset D$. Note that $1\le\psi^2_k\le 3$. 

3. We now very carefully cut off the part of $\psi_k^2$ within a small neighbourhood of $\pd\Sigma_0\times(R_0,\infty)$, and pull back and extend it to weight functions on $M_R$ for some $R\ge R_0$. Observe that away from $\Phi_k(\bigcup_n B(x_m, \omega))$ one has $\abs{\nabla_{g_D}\psi_k^2}\ge C_0$ for some small $C_0$. Set $R\ge R_0$ to be that for all $r\ge R$, $\theta(r)\ge 36/C_0$. Construct a cross-sectional cutoff $\chi\in C_c^\infty(\Sigma_0)$ such that $\chi(y)=0$ on $\Sigma_\epsilon^c$, greater than $1/12$ on $\Sigma_{2\epsilon}$, less than $1/6$ on $\Sigma_{2\epsilon}^c$, and identically $1$ on $\Sigma_{3\epsilon}$. Moreover we ask $\abs{\nabla_{g_D} \chi}\le C_0/72$. See Figure \ref{F1}. Note that we can find such a cutoff because $R$ is taken large enough to give the cross-section enough space to accommodate the tempered decay. Let the weight functions on $\Phi_k^{-1}(\Sigma_0\times(R,\infty))\subset M_R$ be $\psi_k=\Phi_k^*(\chi(y)\psi^2_k)$. As $\chi(y)\psi_k^2$ is identically zero near $\pd\Sigma_0\times (R,\infty)$, we extend $\psi_k$ to all of $M_R$ by $0$. Note that in general $0\le \psi_k\le 3$, and specifically on $\Phi_k^{-1}(\Sigma_{4\epsilon}\times(R,\infty))$ we have $1\le \psi_k\le 3$ .
\begin{figure}[tb]
\resizebox{\linewidth}{!}{
\begin{tikzpicture}
\pgfplotsset{%
every x tick/.style={black, thick},
every y tick/.style={black, thick},
every tick label/.append style = {font=\footnotesize},
every axis label/.append style = {font=\footnotesize},
compat=1.12,
width=25cm,
height=8cm
  }
\begin{axis}[xmin=0, xmax=4.5, ymin=0, ymax=1.2,
 xtick = {0,1,2,3,4}, xticklabels={$\partial\Sigma_0$,$\epsilon$,$2\epsilon$,$3\epsilon$,$4\epsilon$},ytick = {0.125,0.25,1}, yticklabels={1/12,1/6,1},
 scale=0.4, restrict y to domain=-1.5:1.2,
 axis x line=bottom, axis y line= left,
 samples=40]
\addplot[black, samples=100, smooth, domain=-0.2:4.5, thick]
   plot (\x, {0.5*(1+tanh(3.5*(\x-2.2)))});
\draw [dashed] (0,1/8) -- (4.5,1/8);
\draw [dashed] (0,1/4) -- (4.5,1/4);
\draw [dashed] (0,1) -- (4.5,1);
\draw [dashed] (2,0) -- (2,1.2);
\draw [dashed] (3,0) -- (3,1.2);
\end{axis}
\end{tikzpicture}}
\caption{Behaviour of $\chi(y)$ near $\pd\Sigma_0$.}
\label{F1}
\end{figure}
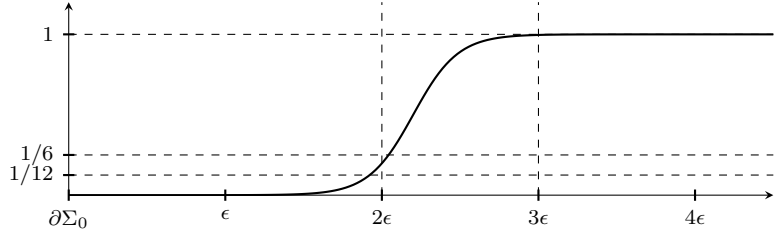

4. We claim that $\psi_k$'s meet our requirement listed in the statement. There is a lower bound for the pushforward map $\|d\Phi_k\|\ge C_1$. Hence we have
\begin{equation}
\abs{\nabla_{g_0}\psi_k}\ge C_1 \abs{\nabla_{g_D}\chi\psi^2_k}.
\end{equation}
Fix $k$ and some point $x\in M_R$ while $x \notin \bigcup_n B(x_m, \omega)$. Note this means $\Phi_k(x)\in \Sigma_0\times(R, \infty)$ and $\Phi_k(x)\notin \Phi_k(\bigcup_n B(x_m, \omega))$. Now set $2\rho=C_0C_1/24$. There are three circumstances depending where $\Phi_k(x)$ is. 

(a) If $\Phi_k(x)\in \Sigma_{3\epsilon}\times(R,\infty)$, the cross-sectional cutoff $\chi$ is identically $1$. We have 
\begin{equation}
\abs{\nabla_{g_D}\chi\psi_k^2}=\abs{\nabla_{g_D}\psi_k^2}\ge C_0,
\end{equation}
and therefore $\abs{\nabla_{g_0}\psi_k}\ge C_0C_1=48\rho\ge 2\rho$. 

(b) If $\Phi_k(x)\in (\Sigma_{2\epsilon}\bigcap\Sigma^c_{3\epsilon})\times(R,\infty)$, the cross-sectional cutoff $\chi>1/12$ is sufficiently large. We have 
\begin{equation}
\abs{\nabla_{g_D} \chi\psi^2_k}\ge \abs{\chi}\abs{\nabla_{g_D} \psi^2_k}-\abs{\psi^2_k}\abs{\nabla_{g_D} \chi}\ge C_0/12-3C_0/72=C_0/24. 
\end{equation}
Here we used the fact that $\psi_k$ is bounded from above by 3. Therefore $\abs{\nabla_{g_0}\psi_k}\ge C_0C_1/24=2\rho$. 

(c) If $\Phi_k(x)\in \Sigma_{2\epsilon}^c\times(R,\infty)$, there is some $l\le n$ such that $\Phi_l(x)\in \Sigma_{4\epsilon}\times(R,\infty)$. From the circumstance (a), we know that $\abs{\nabla_{g_0}\psi_l(x)}\ge 2\rho$, and 
\begin{equation}
\psi_k(x)=\chi\psi^2_k(\phi_k(x))\le 1/2= 1-1/2\le \psi_l(x)-1/2.
\end{equation}
Here we used the fact that $\chi\le 1/6, \psi_k^2\le 3$ and $\psi_l\ge 1$.

The claim has been concluded as above. 
\end{proof}
\begin{remark}
\begin{enumerate}[wide]
  \item We note that the argument is sharp for conic ends, where $\theta(r)=r$. If $\theta'(r)$ is not uniformly bounded, then $\Phi$ loses the quasi-isometric nature, and the argument needs further modifications.
  \item This argument relies, twice when setting up $R_0$ and $R$, on the fact that the cross-sectional space is expanding as $r\rightarrow\infty$. Large $R_0$ makes sure that the $\Phi_k^{-1}(\Sigma_\epsilon\times(R_0, \infty))$ is sufficiently apart from $\Phi_k^{-1}(\pd\Sigma_0\times(R_0, \infty))$, so the critical points inside $\Phi_k^{-1}(\Sigma_\epsilon\times(R_0, \infty))$ will not be pulled to some $x_m$ out of the charted region $\Phi_k^{-1}(\Sigma_0\times(R_0, \infty))$. As in the cylindrical case in Lemma \ref{S3T1} the cross-sectional space is not expanding, this argument does not immediately apply to the cylindrical case. 
\end{enumerate}
\end{remark}

Up to this point, on $(M,g_0)$ we have constructed either a weight function on a cylindrical end, or a finite collection of weight functions on a subconic end, compatible on the end in the way described in Lemma \ref{S3T3}. Now our next proposition provides the final modification of those weights to pull them back to $(M,g)$. 

\begin{proposition}[Construction of Carleman weights]\label{S3T4}
On the prespecified $(M, g)$, there are Carleman weights $\psi_0, \dots, \psi_n\in C_b^\infty(M)$ compatible with control from $(\Omega_{\beta},\Omega_{2\beta})$, in the sense of \eqref{S2LL20}.
\end{proposition}
\begin{proof}
1. We start by reviewing what we have learnt from the previous lemmata. Denote $\{x\in \bigcup_k(\pd M_k\times (a, b)_r)\}$ by $(a,b)$, its closure by $[a,b]$, and $M_0\cup(1, b)$ by $\{r<b\}$ as a matter of convenience. Lemma \ref{S3T1} and Lemma \ref{S3T3} state that we have an uniform $R\ge 1$ with a finite family of weights $\psi_1',\dots,\psi_n'$ on the ends $(R,\infty)$, where $0\le \psi_l'\le 3$ on the end $\pd M_{k_l}\times(R,\infty)$ where it is defined, and $1\le \psi_l'\le 3$ specifically on some $U_l\subset \pd M_{k_l}\times(R,\infty)$, with $(R,\infty)\cap \Omega_{2\beta}^c\subset\bigcup_{l=1}^n U_l$. Moreover there exists $\rho_1>0$, such that for each $l$, at $x\in U_l\setminus \Omega_{2\beta}$, one has $\abs{\nabla_{g_0} \psi_l'}\ge \rho_1$, and if for some $k$ we have $\abs{\nabla_{g_0} \psi_k'}<\rho_1$ then $\psi_l\ge \psi_k+1/2$. 

2. On $(M, g_0)$, we start by constructing a weight on the central compactum, to which the ends on which we have the weights are attached. 
Set $I=432/\rho_1$. On $M_0'=\{r\le R+7I\}$ compact, there exists a Morse function $\psi_0^1$ with finitely many non-degenerate critical points, none of which resides on the boundary $\{r=R+7I\}$. Note that this can be achieved by finding a Morse function on a small closed neighbourhood of $M_0'$, for example $\{r\le R+8I\}$, and find a diffeomorphism to move all critical points not on the new boundary $\{r=R+8I\}$ into $\{r<R+7I\}$ and then restrict the new function to $\{r\le R+7I\}$. Apply a diffeomorphism on $M_0'$, to get $\psi_0^2$ where all critical points of $\psi^1_0$ are moved inside some $B(x_0,\omega)$ given by the Network Control Condition. Note that we can assume without loss of generality that there is a such $B(x_0,\omega)$ inside $M_0'$, by increasing $R$ if needed. Now construct
\begin{equation}
\psi_0'=\frac13+\frac13\left(\max{\psi_0^2}-\min{\psi_0^2}\right)^{-1}\left(\psi_0^2- \min{\psi_0^2}\right).
\end{equation}
Note that $\psi_0'\in [1/3,2/3]$ on $M_0'$, and $\abs{\nabla_{g_0} \psi_0'}\ge \rho_0$ for some positive $\rho_0$ away from $B(x_0,\omega)\subset\Omega_{2\beta}$. Now set 
\begin{equation}
2\rho=\min\{\rho_0, \rho_1/72\}.
\end{equation}
We have $\abs{\nabla_{g_0}\psi_0'}\ge 2\rho$ on $M_0'\setminus \Omega_{2\beta}$, and $\rho_1\ge 144\rho$. These are weight functions on $(M,g_0)$. 

3. We now trim the parts of $\psi_0',\dots,\psi_n'$ inside $[R, R+6I]$, where the supports of those weights could intersect, and extend them to the whole $(M,g_0)$ in a compatible manner. Construct two radial cutoff functions $\chi_0$ and $\chi_1$ in $C^\infty_b(M)$. Let $\chi_0(r)$ be non-increasing, $1$ on $\{r\le R+4I\}$, and $0$ on $[R+5I,\infty)$. Let $\chi_1(r)$ be non-decreasing, $0$ on $\{r\le R+I\}$, and $\chi_1\ge 1/36$ on $[R+2I, \infty)$, and $\chi_1\le 1/18$ on $[R, R+2I]$, and constant $1$ on $[R+3I, \infty)$. Meanwhile we ask $\abs{\pd_r\chi_1}\le \rho_1/216$ on $[R,\infty)$, which makes sense as $I=432/\rho_1$ has been taken large enough. See Figure \ref{F2}. Now set 
\begin{equation}
\psi_0=\chi_0\psi_0',\quad \psi_k=\chi_1 \psi_k',
\end{equation}
extended to the whole manifold $M$ by 0. Note $\psi_0,\dots,\psi_n$ are in $C_b^\infty(M)$.

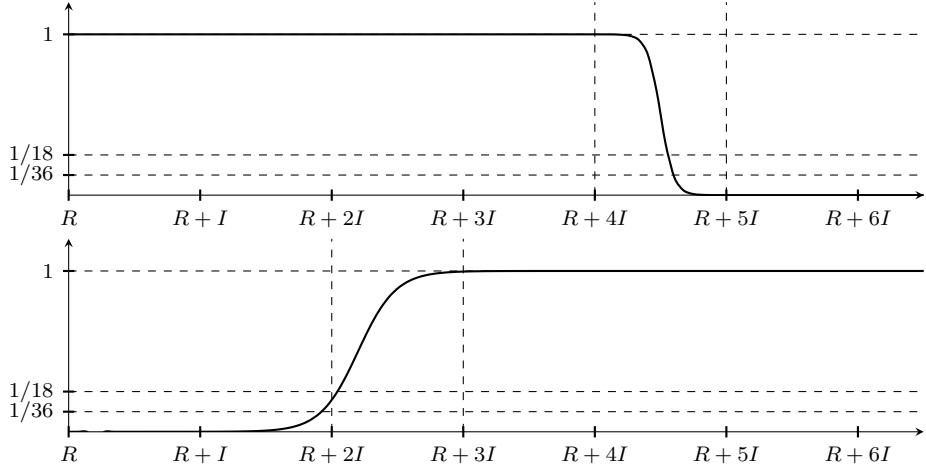
\begin{figure}[tb]
\resizebox{\linewidth}{!}{
\begin{tikzpicture}
\pgfplotsset{%
every x tick/.style={black, thick},
every y tick/.style={black, thick},
every tick label/.append style = {font=\footnotesize},
every axis label/.append style = {font=\footnotesize},
compat=1.12,
width=30cm,
height=8cm
  }
\begin{axis}[xmin=0, xmax=6.5, ymin=0, ymax=1.2,
 xtick = {0,1,2,3,4,5,6}, xticklabels={$R$,$R+I$,$R+2I$,$R+3I$,$R+4I$,$R+5I$, $R+6I$},ytick = {0.125,0.25,1}, yticklabels={1/36,1/18,1},
 scale=0.4, restrict y to domain=-1.5:1.2,
 axis x line=bottom, axis y line= left,
 samples=40]
\addplot[black, samples=100, smooth, domain=-0.2:6.5, thick]
   plot (\x, {0.5*(1+tanh(10*(-\x+4.5)))});
\draw [dashed] (0,1/8) -- (6.5,1/8);
\draw [dashed] (0,1/4) -- (6.5,1/4);
\draw [dashed] (0,1) -- (6.5,1);
\draw [dashed] (4,0) -- (4,1.2);
\draw [dashed] (5,0) -- (5,1.2);
\end{axis}
\end{tikzpicture}}
\resizebox{\linewidth}{!}{
\begin{tikzpicture}
\pgfplotsset{%
every x tick/.style={black, thick},
every y tick/.style={black, thick},
every tick label/.append style = {font=\footnotesize},
every axis label/.append style = {font=\footnotesize},
compat=1.12,
width=30cm,
height=8cm
  }
\begin{axis}[xmin=0, xmax=6.5, ymin=0, ymax=1.2,
 xtick = {0,1,2,3,4,5,6}, xticklabels={$R$,$R+I$,$R+2I$,$R+3I$,$R+4I$,$R+5I$, $R+6I$},ytick = {0.125,0.25,1}, yticklabels={1/36,1/18,1},
 scale=0.4, restrict y to domain=-1.5:1.2,
 axis x line=bottom, axis y line= left,
 samples=40]
\addplot[black, samples=100, smooth, domain=-0.2:6.5, thick]
   plot (\x, {0.5*(1+tanh(3.5*(\x-2.2)))});
\draw [dashed] (0,1/8) -- (6.5,1/8);
\draw [dashed] (0,1/4) -- (6.5,1/4);
\draw [dashed] (0,1) -- (6.5,1);
\draw [dashed] (2,0) -- (2,1.2);
\draw [dashed] (3,0) -- (3,1.2);
\end{axis}
\end{tikzpicture}}
\caption{Behaviour of $\psi_0$ (above) and $\psi_1$ (below) in $[R, R+6I]$.}
\label{F2}
\end{figure}

4. We claim that the $\psi_k$'s satisfy the compatibility conditions \eqref{S2LL20} on $(M,g_0)$ with constant $\tau=1/72$. Keep in mind that $\psi_0\in [1/3,2/3]$ on $\{r\le R+4I\}$. Fix $x\in M\subset \Omega_{2\beta}$. There are three cases.

(a) If $x\in \{r\le R+2I\}$, we have $\chi_0\equiv 1$ and $\chi_1\le 1/18$. For each $k>0$ we have
\begin{equation}
\psi_k=\chi_1\psi_k\le 1/6\le \psi_0-1/72,
\end{equation}
and $\abs{\nabla_{g_0}\psi_0}\ge 2\rho$, no matter how small $\abs{\nabla_{g_0}\psi_k}$ is. Here we used the fact $\psi_k\le 3$ and $\psi_0\ge 1/3$ on $\{r\le R+3I\}$.

(b) If $x\in (R+3I, R+4I]$, we have $\chi_0\equiv 1$ and $1/36\le\chi_1\le 1$. At $x$, for each $k$, if $\abs{\nabla_{g_0}\psi_k'}\ge\rho_1$, then
\begin{equation}\label{S3L6}
\abs{\nabla_{g_0}\psi_k}\ge \chi_1\abs{\nabla_{g_0}\psi_k'}-\psi_k'\abs{\pd_r\chi_1}\ge \rho_1/36-3\rho_1/216=\rho_1/72\ge 2\rho. 
\end{equation}
Here we used that $\psi_k'\le 3$ and $\abs{\pd_r\chi_1}\le \rho_1/216$. Hence if $\abs{\nabla_{g_0}\psi_k}<2\rho$, we have $\abs{\nabla_{g_0}\psi_k'}<\rho_1$. This enables us to invoke the compatibility condition summarised in Step 1. We know now that if $\abs{\nabla_{g_0}\psi_k}<2\rho$, then there is some $l$ such that $x\in U_l\setminus\Omega_{2\beta}$, and at $x$ we have $\psi_l'\ge \psi_k'+1/2$ and $\abs{\nabla_{g_0}\psi_l'}\ge\rho_1$. Moreover $\abs{\nabla_{g_0}\psi_l}\ge2\rho$ and 
\begin{equation}
\psi_l=\chi_1\psi_l'\ge \chi_1\psi_k'+\chi_1/2\ge \psi_k+1/72
\end{equation}
because $\chi_1\ge 1/36$ in this region. 

(c) If $x\in (R+4I, \infty)$, then $\chi_0\le 1$ and $\chi_1\equiv 1$. As in this region we have $\psi_k=\psi_k'$ for each $k>0$, whenever $\abs{\nabla_{g_0}\psi_k}\le 2\rho\le \rho_1$, we have $x\in U_l\setminus \Omega_{2\beta}$ for some $l$ and at $x$ we have $\abs{\nabla_{g_0}\psi_l}\ge \rho_1\ge 2\rho$ and $\psi_l\ge \psi_k+1/2$. Note that no matter how small $\abs{\nabla_{g_0}\psi_0}$ is, we have $\psi_0\le 2/3\le \psi_l-1/72$. This is because $\psi_l\ge1$ on $U_l$. 

5. Up to this point the compatibility conditions \eqref{S2LL20} have been verified on $(M,g_0)$. Now pull back $\psi_k$'s to $(M,g)$ via $\Phi_0$ given in Section \ref{SS12}. Note that $\Phi_0(x)=x$ for each $x\in M$ and therefore one has $\Phi_0(\Omega_{\upsilon})=\Omega_{\upsilon}$ for $\upsilon=\beta,2\beta$, and for $0\le k\le n$, we have $\Phi_0^*\psi_k=\psi_k$ and some constants $C_0>0$ and 
\begin{equation}\label{S3L5}
\abs{\nabla_{g}\Phi_0^*\psi_k}\ge C_0\abs{\nabla_{g_0}\psi_k},
\end{equation}
as $d\Phi_0^{-1}$ is bounded, as in \eqref{S3L8}. Here we directly verify \eqref{S2LL20}, for weights $\Phi_0^*\psi_0,\dots,\Phi_0^*\psi_n$ with constants $2C_0\rho$ and $\tau=1/72$. At each $x\in M\setminus \Omega_{2\beta}$, for each $0\le k\le n$, if we have $\abs{\nabla_g\Phi_0^*\psi_k(x)}< 2C_0\rho$, then by \eqref{S3L5} we know that $\abs{\nabla_{g_0}\psi_k(x)}< 2\rho$. As in Step 4 we have the compatibility conditions on $(M,g_0)$ for $\psi_0,\dots,\psi_n$, there has to be some $l$ such that $\abs{\nabla_{g_0}\psi_l(x)}\ge 2\rho$ and $\psi_l(x)\ge\psi_k(x)+1/72$, and therefore by \eqref{S3L5}, we have
\begin{equation}
\abs{\nabla_{g}\Phi_0^*\psi_l}\ge C_0\abs{\nabla_{g_0}\psi_l}\ge 2C_0\rho, \quad \Phi_0^*\psi_l-\Phi_0^*\psi_k=\psi_l-\psi_k\ge 1/72.
\end{equation}
This concludes the proof.
\end{proof}

\begin{remark}
Similar to the construction in Lemma \ref{S3T3}, we used the fact that we are on an unbounded manifold. Indeed, we need the end to be infinitely long to find a suitable cutoff $\chi_1$, on a compact but huge chunk of the end $[R, R+7I]$, in Step 3. This is to make sure that the entering of the weights on the ends is tempered so it does not dramatically impact the gradient of those weights, as in \eqref{S3L6}. 
\end{remark}

Since we have the compatible weights on $(M,g)$, Theorem \ref{S2T2} holds immediately. 
\begin{proposition}[Global Carleman estimates]\label{S3T5}
On the prespecified $(M, g)$ in Section \ref{SS12}, assume the Network Control Condition $(L,\omega, 2\beta,\{x_m\})$. Then, there exists a constant $C>0$, independent of semiclassical parameter $h>0$ small,
\begin{equation}
\left\|u\right\|_{L^2(M)}\le e^{C/h}\left(\left\|\left(h^2\Delta_g-V(x; h)\right) u\right\|_{L^2(M)}+\left\|u\right\|_{L^2(\Omega_\beta)}\right),
\end{equation}
where $V\in C^\infty_b(M\times[0, h_0])$ is a semiclassical uniformly bounded real potential.  
\end{proposition}
We have two corollaries to use later in the proof of exponential and logarithmic decays for damped Klein-Gordon equations. 
\begin{corollary}[High frequency estimates]\label{S3T6}
On the prespecified $(M, g)$ in Section \ref{SS12}, assume the Network Control Condition $(L,\omega, 2\beta,\{x_m\})$. Let
\begin{equation}
P_h=\left(h^2\Delta_g-1\right)+iha+h^2.
\end{equation}
Then, there exists a constant $C>0$, independent of semiclassical parameter $h>0$ small,
\begin{equation}
\left\|u\right\|_{L^2(M)}\le e^{C/h}\left(\left\|P_h u\right\|_{L^2(M)}+\left\|u\right\|_{L^2(\Omega_\beta)}\right).
\end{equation}
\end{corollary}
\begin{proof}
Let $V(x; h)=1+h^2$. From Proposition \ref{S3T5} we know
\begin{multline}
\left\|u\right\|_{L^2(M)}\le e^{C/h}\left(\left\|\cre{\left(P_h\right)} u\right\|_{L^2(M)}+\left\|u\right\|_{L^2(\Omega_\beta)}\right)\\
\le e^{C/h}\left(\left\|P_h u\right\|_{L^2(M)}+\left\|u\right\|_{L^2(\Omega_\beta)}\right),
\end{multline}
as $\left\|\cre{\left(P_h\right)} u\right\|_{L^2(M)}\le\left\|P_h u\right\|_{L^2(M)}$.
\end{proof}
\begin{corollary}[Low frequency estimates]\label{S3T7}
On the prespecified $(M, g)$ in Section \ref{SS12}, assume the Network Control Condition $(L,\omega, 2\beta,\{x_m\})$. For a fixed $\mu$, let
\begin{equation}
P_\mu=\Delta_g+i\mu a-(\mu^2-1).
\end{equation}
Then, there exists a constant $C>0$ such that
\begin{equation}\label{S3LL10}
\left\|u\right\|_{L^2(M)}\le C\left(\left\|P_\mu u\right\|_{L^2(M)}+\left\|u\right\|_{L^2(\Omega_\beta)}\right).
\end{equation}
\end{corollary}
\begin{proof}
Let $V(x)=h^2\left(1-\mu^2\right)$. From Proposition \ref{S3T5} we know
\begin{multline}
\left\|u\right\|_{L^2(M)}\le e^{C/h}\left(\left\|\cre{\left(h^2P_\mu\right)} u\right\|_{L^2(M)}+\left\|u\right\|_{L^2(\Omega_\beta)}\right)\\
\le e^{C/h}\left(\left\|h^2P_\mu u\right\|_{L^2(M)}+\left\|u\right\|_{L^2(\Omega_\beta)}\right),
\end{multline}
as $\left\|\cre{\left(h^2P_\mu\right)} u\right\|_{L^2(M)}\le\left\|h^2P_\mu u\right\|_{L^2(M)}$. Fix some $h$ small to see \eqref{S3LL10}. 
\end{proof}

\section{Exponential decay of energy}
In this section we aim to show the energy decays exponentially under the Geometric Control Condition in Definition \ref{S1T2}. To do this we need to use Corollary \ref{S3T7} which requires the Network Control Condition. 

\begin{remark}\label{S4T1}
By assuming the Geometric Control Condition $(T,\alpha)$ on $(M,g)$ we can show that the Network Control Condition also holds. Given the Geometric Control Condition, for each $x\in M$, there is some $y\in B(x, T)$ such that $a(y)\ge\alpha$. Cover $(M,g)$ by $\{B(x_m, T)\}$ and let $L=2T$ and $2\beta=\alpha/{2}$, then for each $y\in M$, we have $d(y, \bigcup \{x_m\})\le L$. Meanwhile, as $a\in C^\infty_b(M)$ is uniformly continuous, there is some $\omega>0$ such that $a\ge2\beta=\alpha/2$ on each $B(x_m, \omega)$. Therefore we do have the Network Control Condition $(L, \omega, 2\beta, \{x_m\})$. 
\end{remark}
To characterise the exponential decay, we cite a theorem of \cite{hua85}. See also \cite{gea78,pru84}. 
\begin{theorem}[Gearhart-Pr\"{u}ss-Huang]\label{S4T2}Let $e^{tA}$ be a $C^0$-semigroup in a Hilbert space $X$, and assume there is $C>0$ such that $\|e^{tA}\|_{X\rightarrow X}\le C$ for all $t\ge0$. Then there is $c>0$ for $\|e^{tA}\|_{X\rightarrow X}\le e^{-ct}$ for all $t\ge0$ if and only if $i\setr\cap \sigma(A)=\emptyset$, that is, the spectrum of $A$ has no purely imaginary elements, and
\begin{equation}\label{S4L1}
\sup_{\mu\in\setr}\left\|\left(A-i\mu\right)^{-1}\right\|_{X\rightarrow X}<\infty.
\end{equation}
\end{theorem}
We now give a full proof of Theorem \ref{LT1} concerning the exponential decay of energy, closely following the idea of \cite{bj16}.  
\begin{proof}[Proof of Theorem \ref{LT1}]
1. We start by setting up a proof by contradiction. Let $X=H^1(M)\times L^2(M)$. We will drop $M$ and write $H^1$ and $L^2$ whenever there is no confusion. To show the energy decays exponentially, it suffices to show that the semigroup $e^{tA}$ generated by 
\begin{equation}
A=\begin{pmatrix}
0 & \id\\
-\left(\Delta_g+\id\right) & -a(x)
\end{pmatrix}
\end{equation}
is exponentially stable, in the sense that there exists $C>0$ such that
\begin{equation}
\left\|e^{tA}\right\|_{X\rightarrow X}\le e^{-Ct}
\end{equation}
for each $t\ge 0$. Indeed,
\begin{equation}
E(u)\le \left\|\left(u(t), \pd_t u(t)\right)\right\|_{X}=\left\|e^{tA}\left(u_0, u_1\right)\right\|_{X}\le Me^{-\lambda t}\left\|\left(u_0, u_1\right)\right\|_{X} 
\end{equation}
as claimed. To obtain the exponential stability of $e^{tA}$, it is assumed against \eqref{S4L1}, that $A-i\mu$ is not uniformly bounded from below, that is, there exists a sequence 
\begin{equation}\label{S4LL17}
U_n=(u_n, v_n)\in H^2\times H^1, \quad \left\|U_n\right\|_{X}^2=\left\|u_n\right\|_{H^1}^2+\left\|v_n\right\|_{L^2}^2=1,
\end{equation}
and $\left\{\mu_n\right\}\subset \setr$ such that $\left(A-i\mu_n\right)U_n=o_{X}(1)$. That is, 
\begin{equation}
\begin{cases}
v_n=i\mu_n u_n+o_{H^1}(1)\\
P_{\mu_n}u_n=\left(\Delta_g+\id\right)u_n+a v_n+i\mu_n v_n=o_{L^2}(1),
\end{cases}
\end{equation}
which reduces to 
\begin{equation}\label{S4L3}
\begin{cases}
v_n=i\mu_n u_n+o_{H^1}(1)\\
P_{\mu_n}u_n=\left(\Delta_g+\id\right)u_n+i\mu_n a u_n-\mu_n^2 u_n=o_{L^2}(1).
\end{cases}
\end{equation}
There are two cases: (a) the low frequency case when $\{\mu_n\}$ is bounded; (b) the high frequency case when $\{\mu_n\}$ is unbounded. 

2. Consider the low frequency case (a) and we show there is a contradiction via the low frequency Carleman estimate. As $\{\mu_n\}$ is bounded, by passing through a convergent subsequence one has $\mu_n\rightarrow \mu\in \setr$. By the continuity of the one-parameter family $P_{*}$ we have
\begin{equation}\label{S4L2}
P_{\mu}u_n=\left(\Delta_g+\id\right)u_n+i\mu a u_n-\mu^2 u_n=o_{L^2}(1).
\end{equation}
Here $\{u_n\}$ forms an $o_{L^2}(1)$-quasimode associated with $P_\mu$. Pair \eqref{S4L2} with $u_n$ to see as $n\rightarrow\infty$, 
\begin{equation}
\left\|u_n\right\|_{H^1}=\mu\left\|u_n\right\|_{L^2}+o(1),\quad \left\|\sqrt{a}u_n\right\|_{L^2}=o(1).
\end{equation}
From \eqref{S4L3} we know that $\|v_n\|_{L^2}=\mu \|u_n\|_{L^2}+o(1)$ and $1\equiv\|U_n\|_{X}=\sqrt{2}\mu\|u_n\|_{L^2}+o(1)$. Note that this rules out the possibility that $\mu=0$. Therefore $\|u_n\|_{L^2}=1/\sqrt{2}+o(1)$, as $n\rightarrow\infty$. On $\Omega_\beta$ we have $a\ge \beta$ and then $\|u_n\|_{L^2(\Omega_\beta)}$ is bounded by $\beta^{-1/2}\|\sqrt{a}u_n\|_{L^2(\Omega_\beta)}=o(1)$. The Geometric Control Condition $(T,\alpha)$ implies the Network Control Condition $(L, \omega, 2\beta, \{x_m\})$, as in Remark \ref{S4T1}. Now invoke Corollary \ref{S3T7}. There is some $C>0$ such that for all $n$, we have
\begin{equation}
\left\|u_n\right\|_{L^2(M)}\le C\left(\left\|P_\mu u_n\right\|_{L^2(M)}+\left\|u_n\right\|_{L^2(\Omega_{\beta})}\right).
\end{equation}
Send $n$ to $\infty$, it becomes 
\begin{equation}\label{S4L5}
\frac{1}{\sqrt{2}}\le C\left(o_n\left(1\right)+o_n(1)\right)+o_n(1)=o_n(1),
\end{equation}
which leads to the desired contradiction in low frequencies.

3. Consider the high frequency case (b). As $\{\mu_n\}$ is not bounded, by passing through a subsequence we can assume $\mu_n\rightarrow\pm \infty$. As $A$ is a linear real operator, by symmetry we could assume without loss of generality that $\mu_n=h^{-1}\rightarrow\infty$. The system \eqref{S4L3} is reduced to
\begin{equation}\label{S4L4}
P_h u_h=\left(h^2 \Delta_g-1\right)u_h+ihau_h+h^2u_h=o_{L^2}(h^2)+o_{H^1}(h)
\end{equation}
via reparametrisation by $h$ instead of $n$. We claim that the operator $P_h$ is invertible on $L^2$ with 
\begin{gather}\label{S4L7}
\left\|P_h^{-1}\right\|_{L^2\rightarrow L^2}\le C/h,\\
\label{S4LL13}\left\|P_h^{-1}\right\|_{H^1\rightarrow H^1}\le C'/h,
\end{gather}
within Step 3. 

3a. We set up another proof by contradiction against \eqref{S4L7}, and then establish a commutator argument. Assume against \eqref{S4L7} that there is a family of $\cur{w_h}\subset H^2$ with 
\begin{equation}\label{S4L8}
\left\|w_h\right\|\equiv 1,\quad P_h  w_h=o_{L^2}(h).
\end{equation}
Consider that
\begin{equation}
P_h=\oph(\abs{\xi}^2-1)+ih\oph(a)+\bigo_{H^1_h\rightarrow L^2}(h).
\end{equation}
Pick a symbol $b(x,\xi)\in S^0_u(T^*M)$, to be determined later. Compute the commutator of $\oph(b)$ and $P_h$, 
\begin{equation}\label{S4L9}
\left[\oph\left(b\right), P_h\right]=ih\oph\left(\left\{\abs{\xi}^2, b\right\}\right)+\bigo_{L^2\rightarrow L^2}\left(h^2\right),
\end{equation}
while
\begin{multline}\label{S4LL10}
\left\langle \left[\oph\left(b\right), P_h\right]w_h, w_h\right\rangle=\left\langle \oph\left(b\right) P_hw_h, w_h\right\rangle-\left\langle \oph\left(b\right)w_h, P_h^* w_h\right\rangle\\
=o\left(h\right)+\left\langle \oph\left(b\right)w_h, P_h w_h\right\rangle+\left\langle \oph\left(b\right)w_h, 2iha(x) w_h\right\rangle=-2ih\left\langle \oph\left(ab\right)w_h,  w_h\right\rangle+o\left(h\right).
\end{multline}
Observe from \eqref{S4L9} and \eqref{S4LL10} that
\begin{equation}\label{S4L6}
\left\langle \oph\left(2ab+\left\{\abs{\xi}^2, b\right\}\right)w_h, w_h\right\rangle=o(1).
\end{equation}

3b. Now we show there is a semiclassical concentration phenomenon near the unit speed Hamiltonian flow, and use the Geometric Control Condition to construct an explicit counterexample against \eqref{S4L6}, to conclude \eqref{S4L7}. Construct
\begin{equation}
b(x,\xi)=e^{c(x,\xi)},\quad c(x,\xi)=\frac2T \int_{0}^T\int_0^t \varphi_s^*a\left(x,\xi\right)~ds dt\ge 0.
\end{equation}
Note that $\{\abs{\xi}^2, \varphi^*_s a\}=\pd_\tau \varphi_\tau^* a|_{\tau=s}$ and we can verify that on $(x,\xi)\in\Sigma$ we have
\begin{equation}
2ab+\left\{\abs{\xi}^2, b\right\}=2e^{c(x,\xi)}\ang{a}_T(x,\xi)\ge \alpha>0.
\end{equation}
Now take a smooth microlocal cutoff $\chi\in \fscinf_b(T^*M)$ which is $1$ on $\Sigma=\{|\xi|^2=1\}$, supported inside $\{1/2\le \abs{\xi}\le 2\}$ and is $0$ whenever $2ab+\{\abs{\xi}^2, b\}=2e^{c(x,\xi)}\ang{a}_T(x,\xi)\le \alpha/2$. We claim that $w_h$ is microlocally concentrating near the unit speed set $\{|\xi|^2=1\}$, in the sense that $\left\langle \oph(1-\chi)w_h, w_h\right\rangle=o(1)$, as $h\rightarrow 0$. Note that the semiclassical principal symbol of $P_h$ is $p(x,\xi)=\abs{\xi}^2-1$, which is not 0 on the support of $1-\chi$. Hence
\begin{multline}\label{S4LL12}
\left\langle \oph\left(1-\chi\right)w_h,w_h \right\rangle=\left\langle \oph\left((1-\chi)p^{-1}\right)P_h w_h,w_h \right\rangle\\
+h\left\langle R_{-1}w_h, w_h\right\rangle=\bigo\left(h\right)
\end{multline}
for some $R_{-1}\in \Psi^{-1}_{u, h}$ which is then bounded on $L^2$. Check Appendix \ref{SA} for the class of semiclassical uniform pseudodifferential operators $\Psi^{*}_{u, h}$. Similarly we have
\begin{multline}\label{S4LL11}
\left\langle \oph\left(\left(2ab+\left\{\abs{\xi}^2, b\right\}\right)\left(1-\chi\right)\right)w_h, w_h\right\rangle\\
=\left\langle \oph\left(\left(2ab+\left\{\abs{\xi}^2, b\right\}\right)\left(1-\chi\right)p^{-1}\right)P_h w_h, w_h\right\rangle+h\left\langle R_0w_h, w_h\right\rangle=\bigo(h),
\end{multline}
for some $R_{0}\in \Psi_{u, h}^{0}$ which is bounded on $L^2$. Here we also used the fact that $\oph((2ab+\{\abs{\xi}^2, b\})(1-\chi))\in \Psi_{u, h}^{1}$. From \eqref{S4LL12} and \eqref{S4LL11} we know
\begin{multline}
\left\langle \oph\left(2ab+\left\{\abs{\xi}^2, b\right\}\right)w_h, w_h\right\rangle\\
=\left\langle \oph\left(\left(2ab+\left\{\abs{\xi}^2, b\right\}\right)\chi+\frac{\alpha}2 \left(1-\chi\right)\right)w_h, w_h\right\rangle+\bigo\left(h\right).
\end{multline}
As a symbol of order 0, $(2ab+\{\abs{\xi}^2, b\})\chi+\alpha (1-\chi)/2\ge \alpha/2$ everywhere on $T^*M$. We apply the Garding inequality \ref{A1T1} and see
\begin{equation}
\left\langle \oph\left(\left(2ab+\left\{\abs{\xi}^2, b\right\}\right)\chi+\frac{\alpha}2 \left(1-\chi\right)\right)w_h, w_h\right\rangle\ge C \left\|w_h\right\|^2=C>0, 
\end{equation}
uniformly for small $h$. Therefore
\begin{equation}
\left\langle \oph\left(2ab+\left\{\abs{\xi}^2, b\right\}\right)w_h, w_h\right\rangle\ge C+\bigo \left(h\right).
\end{equation}
This contradicts \eqref{S4L6} immediately. Therefore we have established \eqref{S4L7}.

3c. We want to improve \eqref{S4L7} and get the estimate \eqref{S4LL13} on $H^1$. 
Consider for each $w\in H^1$, we have
\begin{multline}
\left\|P_h^{-1}w\right\|_{H^1}\le\left\|P_h^{-1}w\right\|_{L^2}+\left\|\nabla_{g}P_h^{-1}w\right\|_{L^2}\le Ch^{-1} \left\|w\right\|_{L^2}+\left\|P_h^{-1}\nabla_{g}w\right\|_{L^2}\\
+h^{-1}\left\|\left[h\nabla_g, P_h^{-1}\right]w\right\|_{L^2}\le Ch^{-1}\left(\left\|w\right\|_{L^2}+\left\|\nabla_g w\right\|_{L^2}\right)+C''\|w\|_{L^2}\le C'h^{-1}\left\|w\right\|_{H^1}.
\end{multline}
This is what we need. 

4. We now use our estimates \eqref{S4L7} and \eqref{S4LL13} to finish the proof by contradiction. Rewrite \eqref{S4L4} as 
\begin{equation}\label{S4LL14}
P_h u_h=f_h+o_{L^2}(h^2),
\end{equation}
where $f_h=o_{H^1}(h)$. Note that $P_h^{-1}f_h=o_{H^1}(1)$ because of \eqref{S4LL13}. Now observe 
\begin{multline}\label{S4LL16}
\cre{\left\langle f_h, P_h^{-1}f_h\right\rangle}=\cre{\left\langle P_h \left(P_h^{-1}f_h\right), P_h^{-1}f_h\right\rangle}\\
=\langle \left(h^2\Delta_g -1+h^2\right)P_h^{-1}f_h, P_h^{-1}f_h\rangle=h^2\left\|\nabla_g P_h^{-1}f_h\right\|^2_{L^2}+h^2\left\|P_h^{-1}f_h\right\|^2_{L^2}\\
-\left\|P_h^{-1}f_h\right\|_{L^2}^2=h^2\left\|P_h^{-1}f_h\right\|_{H^1}^2-\left\|P_h^{-1}f_h\right\|_{L^2}^2.
\end{multline}
The second last equality comes from integration by parts. Meanwhile as $f_h=o_{H^1}(h)$ we have $\cre{\langle f_h, P_h^{-1}f_h\rangle}=o(h)\left\|P_h^{-1}f_h\right\|_{L^2}$. Hence
\begin{multline}\label{S4LL15}
\left\|P_h^{-1}f_h\right\|_{L^2}=\frac12\left(o(h)+\sqrt{o(h^2)+4h^2\left\|P_h^{-1}f_h\right\|_{H^1}^2}\right)\\
=h\left\|P_h^{-1}f_h\right\|_{H^1}+o(h)=o(h). 
\end{multline}
By \eqref{S4LL14} we have
\begin{equation}
P_h\left(u_h-P_h^{-1}f_h\right)=P_hu_h-f_h=o_{L^2}(h^2),
\end{equation}
Then $u_h-P_h^{-1}f_h=P_h^{-1}(o_{L^2}(h^2))=o_{L^2}(h)$, that is $u_h=P_h^{-1}f_h+o_{L^2}(h)=o_{L^2}(h)$, as a result of \eqref{S4LL15}. From \eqref{S4L3} we observe that $\left\|v_h\right\|_{L^2}=h^{-1} \left\|u_h\right\|_{L^2}+o(1)=o(1)$. Similarly to \eqref{S4LL16}, we take the real part of the $L^2$-inner product between $P_h u_h$ and $u_h$ to see
\begin{equation}
\left\|u_h\right\|_{H^1}^2=h^{-2}\left(\left\|u_h\right\|_{L^2}^2+o(h)\left\|u_h\right\|_{L^2}\right)=o(1).
\end{equation}
Now we have
\begin{equation}
\left\|U_h\right\|_X^2=\left\|u_h\right\|_{H^1}^2+\left\|v_h\right\|_{L^2}^2=o(1)
\end{equation}
This contradicts our assumption that this term should be constantly $1$, as in \eqref{S4LL17}. 

5. Now we bring together the contradictions in high frequencies and low frequencies to see
\begin{equation}
\sup_{\mu\in\setr}\left\|\left(A-i\mu\right)^{-1}\right\|_{X\rightarrow X}<\infty.
\end{equation}
Hence by Theorem \ref{S4T2}, we conclude that $e^{tA}$ is exponentially stable and the energy decays exponentially. 
\end{proof}

\section{Logarithmic decay of energy}
In this section we aim to show the energy decays logarithmically under the Network Control Condition. In order to characterise the logarithmic decay, we cite \cite[Theorem 3]{bur98}.
\begin{theorem}[Burq]\label{S5T1}Let $A$ be a maximal dissipative operator that generates a contraction $C^0$-semigroup in a Hilbert space $X$ and assume that there exist $C, c>0$ such that $i\mathbb{R}\cap \sigma(A)=\emptyset$, that is, the spectrum of $A$ has no purely imaginary elements, and assume for any $\mu\in \setr$,
\begin{equation}\label{S5L1}
\left\|\left(A-i\mu\right)^{-1}\right\|_{X\rightarrow X}<C e^{c\abs{\mu}}.
\end{equation}
Then for any $k>0$ there is $C_k$ such that for any $t>0$, 
\begin{equation}
\left\|\frac{e^{tA}}{\left(1-A\right)^k}\right\|_{X\rightarrow X}\le \frac{C_k}{\log\left(2+t\right)^k}.
\end{equation}
\end{theorem}
Now we give a proof of Theorem \ref{LT2}. 
\begin{proof}
1. We set up a proof by contradiction against \eqref{S5L1}. Let $X=H^1(M)\times L^2(M)$, and drop $M$ whenever there is no confusion. Assume that for all $c>0$, there exists a sequence of $U_n=(u_n, v_n)\in H^2\times H^1$, $\|U_n\|_{X}=1$, and $\left\{\mu_n\right\}\subset \setr$ such that $\left(A-i\mu_n\right)U_n=o_{H^1\times L^2}(e^{-c/h})$. This is reduced to
\begin{equation}\label{S5L2}
\begin{cases}
v_n=i\mu_n u_n+o_{H^1}(e^{-c/h})\\
P_{\mu_n}u_n=\left(\Delta_g+\id\right)u_n+i\mu_n a u_n-\mu_n^2 u_n=o_{L^2}(e^{-c/h}).
\end{cases}
\end{equation}
Again, as in the proof of Theorem \ref{LT1}, there are two cases: (a) the low frequency case when $\{\mu_n\}$ is bounded; (b) the high frequency case when $\{\mu_n\}$ is unbounded. 

2. Recall that the low frequency case under the Network Control Condition has been dealt with, in Step 2 of the proof of Theorem \ref{LT1}. See \eqref{S4L2} to \eqref{S4L5}. Therefore we have the desired contradiction. It suffices to look at the high frequency case, in which $\{\mu_n\}$ is not bounded, assuming merely the Network Control Condition. 

3. In the high frequency case, we use the high frequency Carleman estimate derived in Corollary \ref{S3T6} to show there is a contradiction. As $A$ is a linear real operator, by passing through a subsequence and by the symmetry, we may assume without loss of generality that $\mu_n=h^{-1}\rightarrow\infty$. Let $P_h=h^2P_\mu$. We reduce \eqref{S5L2} to
\begin{equation}\label{S5L3}
P_h u_h=\left(h^2 \Delta_g-1\right)u_h+ihau_h+h^2u_h=o_{L^2}(h^2 e^{-c/h}).
\end{equation}
From \eqref{S5L2} we know that 
\begin{equation}\label{S5L5}
\|v_h\|_{L^2}=h^{-1}\|u_h\|_{L^2}+o(e^{-c/h}).
\end{equation}
Observe
\begin{multline}\label{S5L4}
\left\langle P_h u_h, u_h\right\rangle=h^2\left\|\nabla_g u_h\right\|_{L^2}^2+h^2\left\|u_h\right\|_{L^2}^2-\left\|u_h\right\|_{L^2}^2+ih\left\langle au_h, u_h\right\rangle \\
=h^2\left\|u_h\right\|_{H^1}^2-\left\|u_h\right\|_{L^2}^2+ih\left\langle au_h, u_h\right\rangle
\end{multline}
as a result of integration by parts. Recall that $\|u_h\|_{L^2}\le \|u_h\|_{H^1}\le 1$, as $U_h$ is normalised. Therefore $\langle P_h u_h, u_h\rangle=o(h^2e^{-c/h})\|u_h\|_{L^2}$. Compare this with \eqref{S5L4} to see 
\begin{gather}\label{S5L6}
\|u_h\|_{H^1}=h^{-1}\|u_h\|_{L^2}+o(h e^{-c/h})\\
\|\sqrt{a}u_h\|_{L^2}=o(h^{\frac12}e^{-c/2h})\left\|u_h\right\|_{L^2}^\frac{1}{2}.
\end{gather}
Bring \eqref{S5L5} and \eqref{S5L6} together to see 
\begin{multline}
1=\left\|U_h\right\|_X^2=\left\|u_h\right\|_{H^1}^2+\left\|v_h\right\|_{L^2}^2=2h^{-2}\left\|u_h\right\|^2_{L^2}+2h^{-1}\left\|u_h\right\|_{L^2}o(e^{-c/h})\\
+o(e^{-2c/h}).
\end{multline}
Therefore
\begin{equation}
\left\|u_h\right\|_{L^2}=\sqrt{\frac{h^2}{2}+o\left(h^2e^{-2c/h}\right)}+o(h e^{-c/h})=\frac{h}{\sqrt{2}}+o(h e^{-c/h}).
\end{equation}
On $\Omega_\beta$ we have $a\ge \beta$ and then
\begin{equation}
\|u_h\|_{L^2(\Omega_\beta)}\le \beta^{-1/2}\|\sqrt{a}u_h\|_{L^2(\Omega_\beta)}=o(h^{\frac12}e^{-c/2h})\left\|u_h\right\|_{L^2}^{1/2}=o(he^{-c/h}).
\end{equation}
Now invoke Corollary \ref{S3T6}. We have for all $h$ small, there is a positive constant $C$ such that
\begin{equation}
\left\|u\right\|_{L^2(M)}\le e^{C/h}\left(\left\|P_h u\right\|_{L^2(M)}+\left\|u\right\|_{L^2(\Omega_\beta)}\right),
\end{equation}
which in our context reads
\begin{equation}
\frac{h}{\sqrt{2}}\le e^{C/h}\left(o\left(h^2 e^{-c/h}\right)+o\left(he^{-c/h}\right)\right)=o\left(he^{(C-c)/h}\right)
\end{equation}
which does not hold for any $c\ge C$. Hence we obtain the contradiction. 

4. We now claim \eqref{S1LL16} by sacrificing regularity for better decay. We have shown in Step 1, 2 and 3 that there are $C, c>0$ such that
\begin{equation}
\left\|\left(A-i\mu\right)^{-1}\right\|_{H^{1}\times L^2\rightarrow H^{1}\times L^2}<C e^{c\abs{\mu}},
\end{equation}
for all $\mu\in \setr$. Invoke Theorem \ref{S5T1} to see
\begin{equation}
\left\|\frac{e^{tA}}{\left(1-A\right)^k}\right\|_{H^{1}\times L^2\rightarrow H^{1}\times L^2}\le \frac{C_k'}{\log\left(2+t\right)^k}.
\end{equation}
As $e^{tA}$ is strongly continuous, it commutes with $(1-A)^k$. Therefore
\begin{multline}
\left\|e^{tA}\left(u_0, u_1\right)\right\|_{H^1\times L^2}=\left\|\left(1-A\right)^k \left(1-A\right)^{-k}e^{tA}\left(u_0, u_1\right)\right\|_{H^1\times L^2}\\
=\left\|\left(1-A\right)^{-k}e^{tA}\left(1-A\right)^k \left(u_0, u_1\right)\right\|_{H^1\times L^2}\\
\le \left\|\left(1-A\right)^{-k}e^{tA}\right\|_{H^1\times L^2\rightarrow H^1\times L^2}\left\|\left(1-A\right)^k \left(u_0, u_1\right)\right\|_{H^1\times L^2}\\
\le \frac{C_k''C_k'}{\log\left(2+t\right)^k}\left\|\left(u_0, u_1\right)\right\|_{H^{k+1}\times H^k}=\frac{C_k}{\log\left(2+t\right)^k}\left\|\left(u_0, u_1\right)\right\|_{H^{k+1}\times H^k}
\end{multline}
because $\left(1-A\right)^k\in \Psi_{u, h}^k(M)$.
\end{proof}

\appendix
\section{Analysis on manifolds of bounded geometry}\label{SA}
We will recall Riemannian geometric terminologies our arguments require in this appendix. Assume our manifold $(M^d, g)$ is smooth, connected, complete, and open, which means non-compact and without boundary. 

There are some invariantly defined classes of functions and operators on $M$. Denote the smooth complex-valued functions on $M$, by $C^\infty(M, \setc)$. Let $L^2_g(M)$ be the class of square-integrable complex-valued functions on $M$ with respect to the density $dg$ induced by the metric. This is a Hilbert space endowed with the inner product 
\begin{equation}
\langle u, v\rangle_{L^2_g(M)}=\int_M u\bar v~dg.
\end{equation}
Moreover $TM$ inherits a bundle metric 
\begin{equation}
\langle X, Y\rangle_{L^2_g(TM)}=\int_M g(X, Y)~dg.
\end{equation}
There is an exterior derivative on 0-forms $d: C^\infty(M)\rightarrow C^\infty(M, T^*M)$. The gradient operator $\nabla_g: C^\infty(M)\rightarrow C^\infty(M, TM)$ is defined as the dual of $d$, where $\nabla_g f$ is defined uniquely by $g(\nabla_g f, X)=df(X)$ for each $X\in TM$. Locally we have
\begin{equation}\label{A1L8}
\nabla_gf=\nabla^i f\nabla_i=\sum_{j}g^{ij}(\pd_j f)\pd_i.
\end{equation}
Define the divergence operator $\nabla_g^*: C^\infty(M, TM)\rightarrow C^\infty(M)$ as the formal $L^2_g$-adjoint of $\nabla_g$, in the sense that 
\begin{equation}
\left\langle \nabla_g^* X, f\right\rangle_{L^2_g}=\int_M g(X, \nabla_g f)~dg. 
\end{equation}
Locally we have
\begin{equation}
\nabla^*_g X=-(\sqrt{g})^{-1}\pd_i\sqrt{g}X^i,
\end{equation}
where $\sqrt{g}=\abs{\det g_{ij}}^{1/2}$. We define the Laplace-Beltrami operator $\Delta_g:C^\infty(M)\rightarrow C^\infty(M)$ as $\Delta_g=\nabla^*_g \nabla_g$, locally given by 
\begin{equation}
\Delta_g=-(\sqrt{g})^{-1}\pd_i\sqrt{g}g^{ij}\pd_j.
\end{equation}
It is a positive symmetric operator on $C^\infty(M)$. 

We follow \cite[Chapter 7]{tri10} to define the manifolds of bounded geometry. For each $p\in M$ let the exponential map at $p$, $\exp_p: T_p M\rightarrow M$ be
\begin{equation}
\exp_p(v)=\gamma(1)
\end{equation}
where $\gamma$ is the unique geodesic such that $\gamma(0)=p$ and $\gamma'(0)=v$. Note that $\exp_p(0)=p$. With a choice of the local orthonormal frame, we identify $T_pM$ by $\setr^d$ isometrically. Given $r>0$ small enough, $\exp_p$ is then a diffeomorphism from $B(0, r)\subset \setr^d$ onto $\Omega_p(r)=\exp_p(B(0, r))=B_M(p, r)\subset M$. Note that $(\Omega_p(r), \exp^{-1}_p)$ is a local cover of $M$ about $p$. We call the corresponding local coordinates geodesic normal coordinates. Note in geodesic normal coordinates about $p=\exp_p(0)$, locally we have
\begin{equation}
g_{ij}(p)=\delta_{ij},\quad \pd_kg_{ij}(p)=0, \quad \Gamma^k_{ij}(p)=0. 
\end{equation}
Let $r_p$ be the supremum of all $r$'s such that $\exp_p$ yields a diffeomorphism. Define the injectivity radius of $M$ by $r_0=\inf_{p\in M} r_p$. We say a manifold $(M,g)$ is a manifold of bounded geometry if (a) the injectivity radius is positive, that is, $r_0>0$; and (b) fixing some $0<r<r_0$, there are constants $C, C_\alpha>0$ such that for each multi-index $\alpha>0$ we have 
\begin{equation}
\det g_{ij}(p)\ge C,\quad \abs{\pd_p^\alpha g_{ij}(p)}\le C_\alpha,
\end{equation}
at each $p\in \Omega_p(r)$ where $\pd_p$ is the differentiation in the geodesic normal coordinates about itself. The uniform boundedness of all derivatives of the metric tensor is equivalent to that of the curvature tensor. See \cite{eic07}. We remark that, on manifolds of bounded geometry, uniform boundedness of derivatives of functions in one choice of geodesic normal coordinates is equivalent to that in another choice. As a result, from now on, we fix $0<r<r_0$ and the bounds on $\pd_p^\alpha u(p)$ could be discussed in arbitrary geodesic normal coordinates of which $p$ is inside. 

We can define the uniformly bounded functions on manifolds of bounded geometry, following \cite{shu92,kor91}. We call a complex-valued $C^k$ function $f: M\rightarrow \setc$ is $C^k$-bounded, denoted $f\in C^k_b(M)$, if for each multi-index $\alpha$ with $\abs{\alpha}\le k$ we have a constant $C_\alpha$ such that $\abs{\pd_p^\alpha f(p)}\le C_\alpha$ for any $p\in M$. This is equivalent to 
\begin{equation}
\abs{\nabla^j f}(p)=\left(\sum_{\alpha_1,\dots, \alpha_j} \left(\nabla^{\alpha_1}\dots \nabla^{\alpha_j}f(p)\right).\left(\nabla_{\alpha_1}\dots \nabla_{\alpha_j}\bar{f}(p)\right)\right)^{\frac12}<C_j
\end{equation}
for each $0\le j\le k$, where $\nabla^{\alpha_l}$'s and $\nabla_{\alpha_l}$'s are respectively the contravariant and covariant derivatives with respect to a local orthonormal frame of $TM$ and $\alpha_1,\dots,\alpha_j$ run through the orthonormal frame. Note that $\abs{\nabla^j f}$ does not depend on the choice of the local orthonormal frame. 
Also write $C^\infty_b(M)=\bigcap_k C^k_b(M)$. We can also define the $L^2$-based uniform Sobolev spaces. Let $H^k(M)$ be the completion of $C_c^\infty(M)$ under the norm
\begin{equation}\label{A1L4}
\left\|f\right\|_{H^k(M)}=\left(\sum_{j=0}^k \int_M \abs{\nabla^j f}^2~dg\right)^{\frac12}.
\end{equation}
Specifically, we have
\begin{equation}
\left\|f\right\|_{H^1(M)}=\left(\left\|f\right\|_{L^2(M)}^2+\left\|\nabla_g f\right\|_{L^2(M)}^2\right)^{\frac12}.
\end{equation}
We look at the partition of unity on manifolds of bounded geometry. We note that there is $\epsilon_0>0$ such that if $0<\epsilon<\epsilon_0$, then there exists a countable cover of $M$ by balls of radius $\epsilon$, say, $B(p_k, \epsilon)$, and moreover the enlarged cover $\{B(p_k, 2\epsilon)\}_k$ has a finite multiplicity. Fix some $\epsilon<r/2$, and in each ball $B(p_k, 2\epsilon)$ we can now use the geodesic normal coordinates about $x_k$. For such cover, there is a partition of unity by functions $\chi_k$, that
\begin{equation}\label{A1L2}
\sum_{k=1}^\infty \chi_k=1
\end{equation}
such that 
\begin{enumerate}[label=(\roman*)]
\item $\chi_k\ge 0$, $\chi_k\in C^\infty_c(M)$, $\supp \chi_k\subset B(p_k, 2\epsilon)$;
\item $\abs{\pd_p^\alpha \chi_k(p)}\le C_\alpha$, for each $p\in M$, in arbitrary geodesic normal coordinates, where $C_\alpha$ does not depend on $k$. 
\end{enumerate}
We say a map $f$ from $M$ to $N$, between two manifolds of bounded geometry, is $C^k$-bounded for $k\ge 1$, denoted $f\in C^k_b(M, N)$, if for each $0\le j\le k-1$ we have $C_j>0$ such that
\begin{equation}
\abs{\nabla^j df}(p)\le C_j
\end{equation}
at all $p\in M$, where $\nabla^j$ is the Levi-Civita connection on $M$ applied $j$ times. The class $C^\infty_b(M, N)$ is defined to be the intersection of all $C^k_b(M, N)$ for integers $k\ge 1$. A $C^\infty_b$-diffeomorphism on $M$ is a bijective map in $C^\infty_b(M, M)$ whose inverse is also in $C^\infty_b(M, M)$. See further details in \cite{eic07}. 

We now define the semiclassical uniform pseudodifferential operators on manifolds of bounded geometry. Those are locally semiclassical pseudodifferential operators, but with some uniform control. We start by defining the residual class, in the manner of \cite{dz19}.
\begin{definition}[Residual class]
Let an $h$-dependent operator $A:C^\infty_c(M)\rightarrow C^\infty(M)$ is said to be in the residual class of the semiclassical uniform pseudodifferential operators, denoted $A\in h^\infty \Psi^{-\infty}_u(M)$ if, its kernel $K_A\in C^\infty(M\times M)$ satisfying
\begin{equation}
\abs{\pd_{p}^\alpha\pd_{q}^\beta K_{A}(p, q)}\le C_{\alpha\beta k}h^k
\end{equation}
for each $k$ and each multi-indices $\alpha,\beta$, any $h\in (0, h_0)$, each $p,q\in \Omega_p(r)$, in arbitrary geodesic normal coordinates, where $C_{\alpha\beta k}$ does not depend on $p$ or $q$. 
\end{definition}
With the residual class we could define the semiclassical uniform pseudodifferential operators. 
\begin{definition}[Semiclassical uniform pseudodifferential operators]
We say an operator $A:C_c^{\infty}(M)\rightarrow C^{\infty}(M)$ is a semiclassical uniform pseudodifferential operator of order $m$, denoted $A\in \Psi^m_{u, h}(M)$, if 
\begin{equation}\label{A1L1}
A=\sum_{k=1}^\infty \chi_k \left(\exp^{-1}_{p_k}\right)^*\oph(a_k)\exp_{p_k}^*\chi_k +\bigo\left(h^{\infty}\Psi^{-\infty}_u\right),
\end{equation}
for some partition of unity $\{B(p_k, 2\epsilon), \chi_k\}_k$ described in \eqref{A1L2}, and $a_k\in S_u^m(\setr^d)$ is a symbol on $\setr^d$ with bounds uniformly in $k$, that
\begin{equation}
\sup_{h\in (0, h_0)}\abs{\pd_\xi^\alpha\pd_x^\beta a_k(x,\xi; h)}\le C_{\alpha\beta} \langle \xi\rangle^{m-\abs{\alpha}},
\end{equation}
for each multi-indices $\alpha,\beta$, and $C_{\alpha\beta}$ independent of $x,\xi,k$. 
\end{definition}

For each $h$-dependent $a(x,\xi; h)\in C^\infty(T^*M)$, we say it is a uniformly bounded symbol of order $m$, denoted $a\in S^m_u(T^*M)$, if for each multi-indices $\alpha,\beta$, there exists constants $C_{\alpha\beta}>0$ such that
\begin{equation}\label{A1L5}
\sup_{h\in (0, h_0)}\abs{\pd_\xi^\alpha\pd_x^\beta a(x,\xi;h)}\le C_{\alpha\beta} \langle \xi\rangle^{m-\abs{\alpha}},
\end{equation}
for all $x\in M$ and all $\xi\in T^*_x M$, in the geodesic normal coordinates near $x$, where the constants do not depend on $x, \xi$. Let $\exp_{p_k}'$ be the lifted diffeomorphism from $T^*(B(0, 2\epsilon))$ to $T^*(B(p_k, 2\epsilon))$ defined via
\begin{equation}
\exp_{p_k}'\left(x,\xi\right)=\left(\exp_{p_k}(x), \left(\exp_{p_k}^{-1}\right)^*\xi\right).
\end{equation}
and $\left(\exp_{p_k}^{-1}\right)'$ being its inverse. We have a principal symbol map such that, given $A\in \Psi^m_{u,h}(M)$, there exists a unique $\sigma_h(A)\in S_u^m(T^*M)$, where in each representation of \eqref{A1L1}, we have 
\begin{equation}
\sigma_h(A)=\sum_{k=1}^\infty \chi_k(p)^2\left(\left(\exp_{p_k}^{-1}\right)'\right)^*a_k^0,
\end{equation}
where $a_k^0\in S^m_u(\setr^d)$ is the principal part of $a_k$. The principal symbol $\sigma_h(A)$ is defined independently of representations \eqref{A1L1}. We have a quantisation map on $M$, that is, $\op:S^m_u(T^*M)\rightarrow \Psi^m_{u, h}(M)$, given by
\begin{equation}
\oph(a)=\sum_k^\infty \chi_k'\left(\exp_{p_k}^{-1}\right)^*\oph\left(\left(\exp_{p_k}'\right)^*\left(\chi_k a\right)\right)\exp^*_{p_k}\chi_k',
\end{equation}
where $\chi_k'\in C_c^\infty(B(p_k, 2\epsilon))$ is 1 on the support of $\chi_k$. Note $\oph(A)$ is a properly supported semiclassical uniform pseudodifferential operator of order $m$, and $\sigma_h(\oph(a))=a^0$, where $a^0\in S^m(T^*M)$ is the principal part of $a$. 

We define the semiclassical Sobolev spaces. Let $\langle hD\rangle^s=\oph(\ang{h\xi}^s)$. For each $s\ge0$, define $H^s_h(M)=\{f\in L^2(M): \langle hD\rangle^sf \in L^2(M)\}$ with norm
\begin{equation}
\left\|f\right\|_{H^s_h}=\left\|\langle hD\rangle^s f\right\|_{L^2}.
\end{equation}
When $h=1$, define $H^s(M)=H^s_h(M)$. We see at integer $s=k$, at fixed $h>0$, the spaces $H^s(M)$, $H^k(M)$ defined in \eqref{A1L4}, and $H^s_h(M)$ coincide, with equivalent norms. Uniformly in small $h$ we have positive constants $C^-, C^+$ for that
\begin{equation}
C^-\left(\left\|f\right\|_{L^2}+h^s\left\|f\right\|_{H^s}\right)\le \left\|f\right\|_{H^s_h}\le C^+\left(\left\|f\right\|_{L^2}+h^s\left\|f\right\|_{H^s}\right),
\end{equation}
relating the semiclassical and non-semiclassical spaces.
 
We list some essential properties of this calculus. We have $h^l\Psi_{u, h}^m(M)\subset \Psi_{u,h}^{m+l}(M)$ for each $l>0$. Note that $h^\infty \Psi_{u}^{-\infty}=\bigcap_k h^k\Psi_{u, h}^{-k}$. Each $A\in \Psi^0_{u,h}(M)$ defines a bounded operator on $L^2(M)$ and each $A\in \Psi^m_{u,h}(M)$ is bounded on $C_b^\infty(M)$ for any $m$, for fixed $h>0$. Also note each $A\in \Psi^m_{u,h}$ is bounded from $H^{k+m}$ to $H^k$, for each $k, k+m\ge 0$. The principal symbol map $\Psi^m_{u,h}(M)\rightarrow S^m_u(T^*M)$ has kernel inside $h\Psi^{m-1}_{u,h}(M)$. For each $A\in \Psi^m_{u,h}(M)$, there exists $a\in S^m_u(T^*M)$ such that
\begin{equation}
A=\oph(a)+\bigo(h^\infty\Psi_u^{-\infty}).
\end{equation}
If $A\in \Psi_{u, h}^{m}$ and $B\in \Psi_{u, h}^{l}$, we have $AB\in \Psi_{u, h}^{m+l}$ and
\begin{gather}
\sigma_h\left(AB\right)=\sigma_h(A)\sigma_h(B)\\
\sigma_h(h^{-1}[A, B])=\frac{1}{i}\left\{\sigma_h(A),\sigma_h(B)\right\}.
\end{gather}

We now discuss uniform ellipticity of those pseudodifferential operators. An operator $A\in \Psi_{u, h}^m(M)$ is called uniformly elliptic if there exists a constant $C$ such that, 
\begin{equation}
\abs{a_{k}(x,\xi; h)}\ge C \ang{\xi}^m,
\end{equation}
for each $x\in B(p_k, 2\epsilon)$ and $h\in (0,h_0)$, and $a_{k}(x,\xi; h)$ as in \eqref{A1L1}. For each $A\in \Psi^m_{u,h}(M)$ that is uniformly elliptic, there exists a parametrix $P\in \Psi_{u,h}^{-m}(M)$ such that $PA-\id, AP-\id\in h^\infty\Psi_u^{-\infty}$. We now prove a weak version of the Garding inequality in our setting. 

\begin{proposition}[Weak G\r{a}rding inequality with truncation]\label{A1T1}
Let $(M,g)$ be a manifold of bounded geometry, without boundary, and assume $W\subset M$ is a possibly empty region. If $W\neq\emptyset$, then let $W_\epsilon$ be $\{p\in M: d(p, M\setminus W)<\epsilon\}$. Given $b\in S^m_u(T^*M)$, with $m\ge 0$, such that there is uniform $\alpha>0$, for any $(x,\xi)\in T^*(M\setminus W)$, $\cre{b(x,\xi)}\ge \alpha \langle \xi\rangle^m$. Then there exists $C>0$ and for arbitrary $h$ small, for any $u\in H^m(M)$ with $u\equiv 0$ on $W$, we have
\begin{equation}
\cre{\left\langle \oph \left(b\right) u, u\right\rangle}\ge C\left(\left\|u\right\|_{L^2(M)}^2+h^{m/2}\left\|u\right\|_{H^{m/2}(M)}^2\right).
\end{equation}
\end{proposition}
\begin{proof}
1. We first show the case when $W=\emptyset$. As $\cre b\ge \alpha \langle \xi\rangle^m\ge 0$ everywhere on the cotangent bundle, let 
\begin{equation}
e(x,\xi)=(\cre\langle \xi\rangle^{-m} b(x,\xi))^{1/2}\ge \alpha^{1/2},
\end{equation}
and we see $E=\oph(e)\in \Psi^{0}_{u, h}$ is uniformly elliptic, hence there is a parametrix $P\in \Psi^0_{u, h}$ with $PE-\id\in h^\infty\Psi^{-\infty}_u$ and
\begin{equation}
\left\|w\right\|_{L^2(M)}\le C\left\|Ew\right\|_{L^2(M)}+\bigo(h^{\infty})\left\|w\right\|_{L^2(M)}
\end{equation}
where we used the fact that $P$ and $PE-\id$ are bounded on $L^2$. Therefore we have
\begin{equation}
\left\|Ew\right\|_{L^2(M)}\ge K\left\|w\right\|_{L^2(M)}
\end{equation}
uniformly for small $h$. Now for each $w\in L^2$, as $\langle \xi\rangle^{-m} \cre b$ is a symbol of order 0, 
\begin{equation}\label{A1L3}
\cre \left\langle\oph(\langle \xi\rangle^{-m} b) w, w \right\rangle =\left\langle Ew, Ew\right\rangle +\bigo\left(h\right)\left\|w\right\|^2_{L^2(M)}\ge \frac{K^2}{2}\left\|w\right\|^2_{L^2(M)}.
\end{equation}
Now let $w=\langle hD\rangle^{m/2} u$ for each $u\in L^2(M)$. Consider
\begin{equation}
\cre\left\langle \oph(b) u, u\right\rangle=\cre\left\langle \oph(\langle \xi\rangle^{-m} b) w, w\right\rangle+\bigo(h) \left\langle G u, u\right\rangle,
\end{equation}
where $G\in \Psi^{m-1}_{u, h}$. Note that
\begin{equation}\label{A1L7}
\left\langle G u, u\right\rangle=\left\langle \left\langle hD\right\rangle^{-\frac{m-1}2} G u, \left\langle hD\right\rangle^{\frac{m-1}2}u\right\rangle=\bigo\left(\left\|u\right\|_{H^{(m-1)/2}_h}^2\right)=\bigo\left(\left\|u\right\|_{H^{m/2}_h}^2\right).
\end{equation}
Together with \eqref{A1L3} we conclude
\begin{multline}\label{A1L6}
\cre\left\langle \oph(b) u, u\right\rangle\ge \left(\frac{K^2}2-\bigo\left(h\right)\right)\left\|u\right\|_{H^{m/2}_h}^2\ge \frac{K^2}4 \left\|u\right\|_{H^{m/2}_h}^2\\
\ge C\left(\left\|u\right\|_{L^2}^2+h^{m/2}\left\|u\right\|_{H^{m/2}}^2\right)
\end{multline}
as claimed.

2. Now take $W$ as described in the statement of this proposition. As $b\in S^m_u(T^*M)$, defined as in \eqref{A1L5}, we have a global constant $C_1>0$ such that
\begin{equation}
\abs{\nabla_x b(x,\xi)}\le C_1\left\langle \xi\right\rangle^m
\end{equation}
at each $x\in M$ in geodesic normal coordinates. Therefore there exist small $\epsilon>0$ such that $\cre b(x,\xi)\ge (\alpha/2) \left\langle \xi\right\rangle^m$ on $W_\epsilon=\{p\in M: d(p, M\setminus W)<\epsilon\}$. There exists a cutoff $\chi_\epsilon\in C^\infty_b(M)$ such that $\chi_\epsilon\equiv 1$ on $W_\epsilon^c$ and supported in $W$. As $b\in S^m_u(T^*M)$ we have a global constant $C_2>0$ such that
\begin{equation}
b(x,\xi)\ge- C_2\left\langle \xi\right\rangle^m
\end{equation}
at each $x\in M$ in geodesic normal coordinates. Now set 
\begin{equation}
b'(x,\xi)=b(x,\xi)+2C_2\left\langle \xi\right\rangle^m\chi(x).
\end{equation}
Note that now $\cre b'\ge \alpha'\langle \xi\rangle^m$ for some $\alpha'>0$, everywhere on $T^*M$. Apply what we have obtained in Step 1. We know from \eqref{A1L6} that
\begin{equation}
\cre\left\langle \oph(b') u, u\right\rangle\ge \left(\frac{K^2}2-\bigo\left(h\right)\right)\left\|u\right\|_{H^{m/2}_h}^2\ge \frac{K^2}4 \left\|u\right\|_{H^{m/2}_h}^2. 
\end{equation}
Note that
\begin{equation}
\oph{b'}=\oph{b}+2C_2\oph{\left(\left\langle \xi\right\rangle^m\chi(x)\right)}+\bigo(h)G
\end{equation}
for $G\in \Psi^m_{u, h}(M)$. As $\left\langle \xi\right\rangle^m\chi(x)\equiv 0$ on $W^c$, hence vanishes on $\supp{u}$, and
\begin{equation}
\oph{\left(\left\langle \xi\right\rangle^m\chi(x)\right)}u=\bigo (h^\infty).
\end{equation}
Therefore for small $h$ we have
\begin{multline}
\cre\left\langle \oph(b) u, u\right\rangle\ge \cre\left\langle \oph(b') u, u\right\rangle-\bigo(h)\left\langle Gu, u\right\rangle\ge \frac{K^2}{8}\left\|u\right\|_{H^{m/2}_h}^2\\
\ge C\left(\left\|u\right\|_{L^2}^2+h^{m/2}\left\|u\right\|_{H^{m/2}}^2\right),
\end{multline}
as a result of \eqref{A1L7}.
\end{proof}

\bibliography{robib}
\bibliographystyle{amsalpha}

\end{document}